\documentclass[arxiv,preprint,11pt,leqno]{imsart}
\usepackage[margin=2.5cm]{geometry}

\RequirePackage[OT1]{fontenc}
\RequirePackage{amsthm,amsmath,amssymb,dsfont,amsfonts}
\RequirePackage[numbers]{natbib}
\RequirePackage[pdftex,colorlinks]{hyperref}
\usepackage{mathrsfs,enumerate,bbm,latexsym,graphicx}

\usepackage{enumerate}

\newenvironment{enumerateA}{\begin{enumerate}[\bfseries\upshape(A)]}{\end{enumerate}}

\newtheorem{thm}{Theorem}[section]
\newtheorem{lem}[thm]{Lemma}

\newtheorem{prop}[thm]{Proposition}

\newtheorem{rem}[thm]{Remark}

\newtheorem{conj}[thm]{Conjecture}   
\newcommand{\tgo}{{\tilde{\omega}}}
\renewcommand{\le}{\leqslant} 
\renewcommand{\ge}{\geqslant}
\renewcommand{\leq}{\leqslant} 
\renewcommand{\geq}{\geqslant} 

\newcommand{\ra}{\rangle}
\newcommand{\la}{\langle}

\newcommand{\ind}{\mathds{1}}
\newcommand{\eps}{\varepsilon}

\newcommand{\abs}[1]{\left\vert#1\right\vert}

\newcommand{\ie}{\emph{i.e.,}}

\def\qed{ \hfill $\blacksquare$}  
\let\ga=\alpha \let\gb=\beta \let\gc=\gamma \let\gd=\delta 
     \let\gl=\lambda        \let\go=\omega  \let\gr=\rho \let\gs=\sigma  
  \let\gz=\zeta
 \let\gD=\Delta   
\let\gO=\Omega           
                                
\newcommand{\cA}{\mathcal{A}}\newcommand{\cB}{\mathcal{B}}

\newcommand{\cL}{\mathcal{L}}
\newcommand{\cN}{\mathcal{N}}
\newcommand{\cP}{\mathcal{P}}

\newcommand{\cW}{\mathcal{W}}
\newcommand{\cZ}{\mathcal{Z}}  

\newcommand{\vB}{\mathbf{B}}
\newcommand{\vE}{\mathbf{E}}

\newcommand{\vP}{\mathbf{P}}

\newcommand{\vs}{\mathbf{s}}

\newcommand{\mvt}{\boldsymbol{t}}
\newcommand{\mvw}{\boldsymbol{w}}\newcommand{\mvx}{\boldsymbol{x}}

\newcommand{\bE}{\mathbb{E}}

\newcommand{\bN}{\mathbb{N}}
\newcommand{\bP}{\mathbb{P}}\newcommand{\bR}{\mathbb{R}}

\newcommand{\bV}{\mathbb{V}}
\newcommand{\bZ}{\mathbb{Z}}        

\newcommand{\dR}{\mathds{R}}

\newcommand{\dZ}{\mathds{Z}} 

\newcommand{\sA}{\mathscr{A}}

\newcommand{\sP}{\mathscr{P}}

\newcommand{\sS}{\mathscr{S}}

\DeclareMathOperator{\E}{\mathds{E}}
\DeclareMathOperator{\pr}{\mathds{P}}
\DeclareMathOperator{\sgn}{sgn}
\DeclareMathOperator{\var}{Var}

\newcommand{\dd}{\text{\rm d}}  

\begin{document}

\begin{frontmatter}
\title{High temperature limits  for $(1+1)$-dimensional directed polymer with heavy-tailed disorder.}
\runtitle{Heavy tailed random polymer}

\begin{aug}
\author{\fnms{Partha S.} \snm{Dey}\thanksref{m1}\ead[label=e1]{psdey@illinois.edu}\ead[label=u1,url]{http://math.uiuc.edu/$\sim$psdey/}},
\author{\fnms{Nikos} \snm{Zygouras}\thanksref{m2,t2}\ead[label=e2]{N.Zygouras@warwick.ac.uk}\ead[label=u2,url]{http://www2.warwick.ac.uk/fac/sci/statistics/staff/academic-research/zygouras/}}

\thankstext{t2}{Funded by EPSRC grant EP/L012154/1.}

\runauthor{Dey and Zygouras}

\affiliation{University of Illinois at Urbana-Champaign\thanksmark{m1} and University of Warwick\thanksmark{m2}}

\address{Department of Mathematics \\ University of Illinois at Urbana-Champaign\\ 1409 W.~Green St \\ Urbana, IL 61801 USA\\
\printead{e1}\\
\printead{u1}}
\address{Department of Statistics \\ University of Warwick\\ Gibbet Hill Road \\ Coventry, CV4 7AL, UK\\
\printead{e2}\\
\printead{u2}}
\end{aug}

\begin{abstract}
The directed polymer model at intermediate disorder regime was introduced by Alberts-Khanin-Quastel~\cite{AKQ12}. It was proved that at inverse temperature $\beta n^{-\gamma}$ with $\gamma=1/4$ the partition function, centered appropriately, converges in distribution and the limit is given in terms of the solution of the stochastic heat equation. This result was obtained under the assumption that the disorder variables posses exponential moments, but its universality was also conjectured under the assumption of six moments. We show that this conjecture is valid and we further extend it by exhibiting classes of different universal limiting behaviors in the case of less than six moments. We also explain the behavior of the scaling exponent for the log-partition function under different moment assumptions and values of $\gamma$.
\end{abstract}



\end{frontmatter}

\section{Introduction}\label{sec:int}
\subsection{The model}
We consider the $(1+1)$-dimensional directed polymer in i.i.d.~random environment with high temperature. In particular, let $\gO:=\{\go_{v}\mid v\in\dZ^{2}\}$ be a collection of i.i.d.~random variables indexed by the vertices of $\dZ^{2}$. We will denote their joint law by $\bP$ and the corresponding expectation by $\bE$. Let $\vP_n(\cdot)$ be the measure corresponding to a nearest-neighbor simple random walk, starting at the origin at time $0$ and run up to time $n$. We will denote the set of nearest-neighbor paths by $\sS_0^n=\{((i,s_i))_{i=0}^n \mid s_{0}=0, \abs{s_{i}-s_{i-1}}=1, 1\le i\le n\}$) and by $\la\cdot\ra$ the expectation w.r.t.~$\vP_n$.  The energy of a path $\vs=((i,s_i))_{i=0}^n\in\sS_0^n$ is defined as 
$
H^\go(\vs):=\sum_{i=1}^{n} \go_{i,s_i}
$
and the polymer measure $\vP_{n,\gb_n}$ on $\sS_{0}^{n}$ is given by
\begin{align}
	\frac{\dd\vP_{n,\gb_n}}{\dd\vP_n}(\vs) = (Z^{\,\go}_{n,\gb_n})^{-1}\exp(\gb_n H^\go(\vs)), \quad \vs\in\sS_0^n,
\end{align}
where $Z^{\,\go}_{n,\gb_n}$ is the partition function
\begin{align}
	Z^{\,\go}_{n,\gb_n} :=\la\exp(\gb_n H^\go)\ra = 2^{-n}\sum_{\vs\in\sS_0^n} \exp(\gb_n H^\go(\vs))
\end{align}
and $\gb_n> 0$ is the inverse temperature that we will allow to depend on $n$. In particular, we will consider dependencies that will make $\gb_n$ go to zero, as $n$ tends to infinity, thus considering a {\it high temperature regime.}
The expectation with respect to the polymer measure $\vP_{n,\gb_n}$ will be denoted by $\vE_{n,\gb_n}$. 
\vskip 2mm
The directed polymer model was introduced in \cite{HH85} as a model for the interface of the two-dimensional Ising model with random interactions. Therein, numerical evidence was provided indicating interesting 
super diffusive fluctuation exponents of the interface. Soon after, a physical link to the fluctuation theory of Stochastic Burgers equation \cite{FNS77} was provided in \cite{HHF85}.
It is now well established that the directed polymer model is linked to the Kardar-Parisi-Zhang (KPZ) equation, since the logarithm of the partition function $Z^{\,\go}_{n,\gb_n}$ can be viewed as a discretization of the Hopf-Cole solution to the KPZ equation. Therefore, it provides a rigorous path to verify the predictions made by Kardar-Parisi-Zhang on the fluctuation exponents of the celebrated KPZ universality class.\vspace*{1mm}

In rough terms, the fluctuation exponent $\chi\in[0,1]$  characterizes the fluctuations of $\log Z^{\,\go}_{n,\gb_n}$ in the sense that
\[
\abs{ \log Z^{\,\go}_{n,\gb_n} - \E \log Z^{\,\go}_{n,\gb_n}}  \approx n^{\chi+o(1)}, \qquad \text{for}\,\,n\to\infty.
\]
The transversal exponent $\xi\in[0,1]$ characterizes the fluctuations of the end-point $(n,s_n)$ of a path $\vs\in\sS_0^n$ chosen from the measure $\vP_{n,\gb_n}$, that is
$$
 	\E\vE_{n,\gb_n}\abs{s_n}\approx n^{\xi+o(1)}, \qquad \text{for}\,\,n\to\infty.
$$
 So far, work in understanding the fluctuation exponents of the directed polymer has been constrained to the case where the random variables $\go$ have exponential moments. The prediction, when $\gb_n$ equals a constant $\gb$, is that $\chi=1/3$ and $\gz=2/3$. This has been confirmed, so far, only for the so-called log-gamma polymer, where $\exp(\,-\go)$ has gamma distribution \cite{S12}. In fact, in this case, the full scaling limit of the partition function, constrained to a fixed end point, e.g. $s_n=0$, is established to obey the Tracy-Widom GUE law, \cite{COSZ14}, \cite{OSZ14}, \cite{BCR13}. More precisely, if we centre the point-to-point partition function $\log Z_{n,\gb_n}^\go(0)$ by its free energy $f(\gb):=\lim_{n\to\infty} n^{-1}\log Z_{n,\gb_n}^\go(0)$ and scale with $c(\gb)n^{1/3}$ ($c(\gb)$ being a specific constant), then
  \begin{equation}\label{TWGUE}
  \frac{\log Z_{n,\gb}^\go(0) -n f(\gb)}{c(\gb) n^{1/3}} \xrightarrow[n\to\infty]{(d)} F_{GUE}.
  \end{equation}
 Besides the log-gamma polymer, similar scaling behavior has been established only for a handful of polymer models. Namely the continuum polymer (which is directly related to the Hopf-Cole solution of the KPZ equation) \cite{SS10}, \cite{ACQ11} and the O'Connell-Yor semi-discrete polymer \cite{O12}, \cite{BC14}. However, the scaling limit is conjectured to hold universally, independently of the particular distribution, as long as it possesses exponential moments. The only non-universal constants will be the free energy $f(\gb)$ and the scaling constant $c(\gb)$. In \cite{BBP07}, based on a Flory type argument, it is claimed that the $1/3,2/3$ exponents should be valid as long as disorder possesses more than five moments. Furthermore, based on numerical evidence, it is claimed therein, that the same Tracy-Widom limit theorem \eqref{TWGUE} should be valid. On the other hand, recent numerical studies \cite{GLBR14} have indicated that when disorder fails to have a fifth moment, then the $1/3, 2/3$ exponents should be replaced by exponents, which depend on the tails of the disorder. Furthermore, some guesses of the nature of the limit laws are presented, although a concrete guess is still elusive.
\vskip 2mm
 In \cite{AKQ10} and \cite{AKQ12} the notion of weak universality was introduced. In that work, the authors considered the case where $\beta_n=\beta n^{-1/4}$, i.e. a high temperature regime, and showed that
 under the assumption of exponential moments
 \begin{equation}\label{AKQeq}
 \log  Z^{\,\go}_{n,\gb_n} -n\gl(\gb_n) \xrightarrow[n\to\infty]{(d)} \log\cZ_{\sqrt{2}\gb},
 \end{equation}
   where $\cZ_{\sqrt{2}\gb}$ is the solution of the stochastic heat equation and $\gl(\gb_n):=\log \bE[e^{\gb_n \go}]$. Moreover, it was conjectured therein that the same limit behavior should be valid under only the assumption that disorder possesses more than six moments. Notice that in the case of finite moments $\gl(\gb_n)$ is not defined and therefore a different centering constant would be necessary.
\vskip 2mm
In this article we prove this conjecture. Moreover, motivated by a Flory type argument in \cite{BBP07}, we show that this conjecture is part of a larger picture. The latter is described by a phase diagram for the values of the exponents $(\chi,\xi)$ depending on $(\gc,\ga)$ where $\gb_{n}\approx \gb n^{-\gc}, \gc\ge 0$ and the disorder satisfies $\pr(\go>x)=x^{-\ga+o(1)}$, as $x\to\infty$ for  some $\ga>0$. Let us also mention that in  \cite{AKQ10}, based on Airy process heuristic, it was conjectured that for disorder with exponentially decaying tails, i.e. $\ga=\infty$,  and for $\gc\in[0,1/4]$ the scaling exponents should interpolate linearly between the Gaussian and the KPZ exponents like $\chi=(1-4\gc)/3,\xi=2(1-\gc)/3$. This conjecture, which also fits inside our picture, was recently proved in \cite{FSV14} for the stationary version of O'Connell-Yor, semi-discrete  polymer in Brownian environment. Another earlier work, which also fits the picture is \cite{AL11}, where the authors proved that when $\ga\in (0,2)$ and $\gc=2/\ga-1$ one has $\chi=\xi=1$ (see also \cite{HM07} for the corresponding zero-temperature result).
\vskip 2mm 
Roughly speaking, the picture we propose (see Figure~\ref{fig1}) can be described as follows: given an exponent $\xi\in [1/2,1]$ there exists in the $(\ga,\gc)$ diagram a ``level-curve" 
\begin{align*}
 \xi=
 \begin{cases}
 \displaystyle \frac{1+\ga(1-\gc)}{2\ga-1}  ,&\text{for} \quad \ga\le \frac{5-2\gc}{1-\gc} \\
 \displaystyle &\\
 \displaystyle  \frac{2(1-\gc)}{3}                ,&\text{for} \quad  \ga\ge \frac{5-2\gc}{1-\gc} 
\end{cases}
\end{align*}
along which the polymer in heavy tail disorder with ``$\ga$ moments" and at inverse temperature $\gb_n=\gb n^{-\gamma}$ has transversal fluctuation exponent $\xi$. We will present this in more detail in the next section. 
\vskip 2mm
In this article we prove the validity of this picture on the $\xi=1/2$ regime. In other words we rigorously identify the so called {\it weak disorder regime}, where the polymer behaves diffusively.  Moreover,  in this regime, we identify the scaling limit of the partition function and we see that three different scaling limits exist within three sub-regimes. The three different limit behaviors are related to the Hopf-Cole solution of KPZ, to Gaussian and to Poissonian, respectively for $\ga\geq 6, \ga\in(2,6)$ and $\ga\in (1/2,2)$. These different behaviors are linked to the impact of the `large' weights, which is greater the smaller $\alpha$ gets.
   Before presenting the theorems, let us present our assumptions that will be followed, throughout the article:
\begin{enumerateA}
\item\label{defA} The cumulative distribution function $F(x):=\pr(\go\le x)$ has regularly varying right tail. In other words, we assume that 
\[
	\bar{F}(x):=1-F(x)=x^{-\ga}L(x),\quad \text{ for all } \,x>0,
\] 
for some $\ga>0$ and some slowly varying function $L(x)$, i.e. for any $t>0$ it holds that $L(tx)/L(x)\to 1 $, as $x\to\infty$. Moreover, the choice of $L(\cdot)$ is such that $\bar{F}(x):=x^{-\ga}L(x)$ stays bounded, defining an honest cumulative distribution function.\vspace*{3mm}

\item\label{defB} When $\ga> 2$, we will also assume that 
\begin{align*}
\E[\go]=0,\quad\E[\go^2]=1.
\end{align*}

\item\label{defC} When $\ga\le 2$, we will also assume that 
\begin{align*}
F(-x)=(c_-+o(1))\bar{F}(x),\quad \text{ as } \,x\to\infty,
\end{align*}
for some $c_-\ge 0$. In other words, the left tail is dominated by the right tail when $\ga\le 2$.
\end{enumerateA}
\smallskip
For $t>1$, we  define 
\begin{align}\label{mt}
\text{ the real number }m(t):=\inf\{x\mid \bar{F}(x)\le 1/t\}.
\end{align}
Clearly, $m(t)=t^{1/\ga}L_{0}(t)$ as $t\to\infty$ for some slowly varying function $L_{0}$. For $x\in\dR$, define $x_{+}=\max\{x,0\}$ and $x_{-}=\max\{-x,0\}$.  
\vskip 2mm
We will prove the following:		
\vskip 2mm	
\begin{thm}[$\ga\ge 6,\gc\ge 1/4$]\label{thm:14}
	Assume that the weights satisfy assumptions~\eqref{defA} for some $\ga\ge 6$ and~\eqref{defB}.
	\begin{itemize}
	\item[$\bullet$]
	 Let $\gb_{n}$ be a sequence of real numbers with $\gb_{n}n^{1/4}\to \gb$ as $n\to\infty$, for some $\gb\in(0,\infty)$. Then 
	\[
		\log Z_{n,\gb_n}^{\go} - n\log\E\Big(e^{-\gb_n\go_-} + \sum_{i=1}^{4}\frac{\gb_n^i}{i!}\go_+^i \Big) \xrightarrow[n\to\infty] {(d)}\log \cZ_{\sqrt{2}\gb},
	\]
	where 
	\begin{align}\label{Wiener}
		\cZ_{\sqrt{2}\gb}:= 1+\sum_{k=1}^{\infty} (\sqrt{2}\gb)^k \int_{\gD_k} \int_{\dR^k} \prod_{i=1}^k 
		\gr(t_i-t_{i-1},x_{i}-x_{i-1}) W(\dd t_{i}\, \dd x_i ).
	\end{align}
	with $W(\dd t\,\dd x)$ is a white noise on $\dR_{+}\times \dR$ (formally, a Gaussian process with covariance  given by $\E[W(t,x) W(s,y) ]= \gd(t-s)\gd(x-y)$), $\gD_{k}=\{0=t_{0} < t_{1} < t_{2}<\cdots < t_{k} \le 1\}$ is the $k$-dimensional simplex, $x_{i}\in\dR$ with $x_{0}=0$, and $\rho$ is the standard Gaussian heat kernel
	\[
		\gr(t,x) = \frac{1}{\sqrt{2\pi t}} e^{-x^2/2t}, \qquad t\in(0,1), \,x\in\dR. 
	\]
	
\item[$\bullet$] If $\gb_{n}n^{1/4}\to 0$ as $n\to\infty$, then
\[
	\frac{1}{\gb_n n^{1/4}}\Bigl(\log Z_{n,\gb_n}^{\go} - n\log\E\Big(e^{-\gb_n\go_-} + \sum_{i=1}^3 \frac{\gb_n^i}{i!}\go_+^i \Big) \Bigr) \xrightarrow[n\to\infty]{(d)} \cN(0,2\pi^{-1/2}) ,
\]
where $\cN(0,2\pi^{-1/2})$ is the normal distribution with variance $2\pi^{-1/2}$.
\end{itemize}
\end{thm}	
	\vskip 2mm
	Let us remark that $\cZ_{\sqrt{2}\gb}$, which here is given in terms of a Wiener chaos expansion, is in fact the mild solution to the one dimensional stochastic heat equation, with flat initial condition
	\begin{equation*}
\begin{cases}
\partial_t  u =\frac{1}{2}\Delta \, u +\sqrt{2}\gb \, \dot W\,u & \\
u(0, \cdot ) =1&
\end{cases} \,.
\end{equation*}
When $\alpha\leq 6$, we observe the following behavior
\begin{thm}[$2<\ga\le 6, \gc \ge 3/2\ga$ ]\label{thm:gauss}
Assume that the weights satisfy assumption~\eqref{defA} for some $\ga\in(2,6]$ and \eqref{defB}. Let $\gb_{n}$ be a sequence of real numbers such that $\gb_{n} m(n^{3/2})\to \gb$ as $n\to\infty$, for some $\gb\ge 0$. For $\ga=6$, we also assume that {$\gb_nn^{1/4}\to 0$.} Then 
\begin{align*}
	\frac{1}{\gb_n n^{1/4}}\Bigl(\log Z_{n,\gb_n}^{\go} - n\log\E\Big(e^{-\gb_n\go_-} + \sum_{i=1}^2 \frac{\gb_n^i}{i!}\go_+^i & + \frac{\gb_n^3}{3!}\go_+^3\ind_{\ga>3}\Big) \Bigr)
	 \xrightarrow[n\to\infty]{(d)} \cN(0,2\pi^{-1/2}) .
\end{align*}
\end{thm}

To state our result for $\ga\in(\frac12,2)$, we need to consider a Poisson point process $\cP$ on $(w,t,x)\in \sS:=\dR\times [0,1]\times \dR$ with intensity measure $\eta(\dd w \dd t \dd x) =\frac12\ga |w|^{-1-\ga}(\ind_{w>0}+c_-\ind_{w<0})\dd w \dd t \dd x$. Recall that $\rho$ is the standard Gaussian heat kernel
$
\gr(t,x) = (2\pi t)^{-1/2} \exp(-x^2/2t), x\in \dR, t\in [0,1]. 
$ 
We define the random variables
\begin{align*}
\cW_{0}^{(\ga)} =
\begin{cases}
\qquad\qquad\qquad\qquad\qquad\qquad\quad\ \int_{\sS} w \rho(t,x) (\cP-\eta)(\dd w \dd t \dd x) &\text{ if } \ga\in (1,2)\\
\int_{\sS\cap\{ |w|>1\}} w \rho(x,t) \cP(\dd w \dd t \dd x)  +\int_{\sS\cap\{ |w|\le 1\}} w \rho(t,x) (\cP-\eta)(\dd w \dd t \dd x) &\text{ if } \ga=1\\
\int_{\sS} w\rho(t,x)\cP(\dd w \dd t \dd x) &\text{ if } \ga\in (0,1).
\end{cases}
\end{align*}
and for $\gb\in(0,\infty)$,
\[
\cW_{\gb}^{(\ga)} 
= \frac{1}{\gb} \int_{\sS} (e^{\gb w}-1-\gb w)\rho(t,x)\cP(\dd w \dd t \dd x) + \cW_{0}^{(\ga)}.
\]
The random variables $\cW_{\gb}^{(\ga)} $ are well defined, as stated in the following lemma, whose proof is given in Appendix~\ref{sec:aux}:

\begin{lem}\label{PoissonInt}
For every $\gb\ge 0, \ga\in(1/2,2)$, the random variables $\cW_{\gb}^{(\ga)}$ are finite a.s. Moreover, $\cW_{0}^{(\ga)}$ has stable distribution with characteristic function given by 
\[
\E^{\cP}\exp(iy\cW_{0}^{(\ga)})=
\begin{cases}
 \exp\Big(\int_{\sS} (e^{iy w \rho(t,x)}-1-iy w \rho(t,x)) \,\, \eta(\dd w \dd x\dd t)\,\Big) & \text{ if } \ga\in (1,2) \\
\exp\Big(\int_{\sS\cap \{|w|>1 \} } (e^{iy w \rho(t,x)}-1)  \eta(\dd w \dd x\dd t)                                    & \\
\qquad\qquad+ \int_{\sS \cap \{|w|\le 1\}} (e^{iy w \rho(t,x)}-1-iy w \rho(t,x)) \,\, \eta(\dd w \dd x\dd t)\,\Big)   &\text{ if } \ga=1 \\
\exp\Big(\int_{\sS} (e^{iy w \rho(t,x)}-1)  \eta(\dd w \dd x\dd t)\,\Big) & \text{ if } \ga\in (0,1) ,   
\end{cases}
\]
for $y\in\dR$, where $\E^{\cP}$ denotes expectation with respect to the Poisson process $\cP$. 
\end{lem}
One can explicitly evaluate the integrals in the exponent to calculate the characteristic function of  $\cW_{0}^{(\ga)}$, as presented in the proof of Lemma~\ref{PoissonInt}. We can now state the result in the regime $\ga\in(1/2,2)$
\begin{thm}[$1/2<\ga<2, \gc\ge 3/2\ga$]\label{thm:heavy}
Assume that the weights satisfy~\eqref{defA} for some $\ga\in(1/2,2)$ and~\eqref{defC} for some $c_-\ge 0$. Assume that $\E[\go]=0$ when $\ga>1$. Let $\gb_{n}$ be a sequence of real numbers such that $\gb_{n} m(n^{3/2}) $ converges to $\gb\in [0,\infty)$, as $n\to\infty$. Then 
\[
	\frac{\sqrt{n}}{\gb_n m(n^{3/2})}\Bigl(\log Z_{n,\gb_n}^{\go} - n\gb_n\E\big[\go\ind_{|\go|\le m(n^{3/2})}\big]\,\ind_{\ga=1}\Bigr) \xrightarrow[n\to\infty]{(d)} 2\cW_{\gb}^{(\ga)} .
\]
\end{thm}

\begin{rem}
In the case of $\ga=2$, i.e. $\bar{F}(x)=x^{-2}L(x)$ for some slowly varying function $L(\cdot)$ at inifinity, one can prove that for $\gb_{n}$ ``sufficiently small'' there exists  a sequence $a_{n}$ such that
\begin{align}\label{alpha2}
a_{n}\Bigl(\log Z_{n,\gb_n}^{\go} -  n\log\E(e^{-\gb_n\go_-} + \gb_n\go_+ )\Bigr) \xrightarrow[n\to\infty]{(d)} \cN(0,\gs^{2})
\end{align}
for some constant $\gs^{2}>0$. In particular, define the function $H(x)=\E[\go^{2}\ind_{\{|\go|\le x\}}]$ for $x>0$, which is slowly varying at infinity. Define $\hat{m}(t):=\inf\{x\mid x^{-2}H(x)\le 1/t\}$ for $t>1$. Under the assumption that $$
\gb_{n}\hat{m}(n^{3/2})\to\gb \text{ as } n\to\infty \text{ for some } \gb\in[0,\infty)
$$ 
one can prove that \eqref{alpha2} holds with $c_{n}=\sqrt{n}/\big(\gb_{n}\hat{m}(n^{3/2})\big)$. The proof combines techniques used in the proof of Theorem~\ref{thm:heavy} and the method used in the proof of Theorem~4.17 in Kallenberg~\cite{K02} regarding the central limit behaviour of renormalised sums of random variables that barely fail to have second moment. Since this is technical, we prefer to omit the details for simplicity.
\end{rem}

For the sake of completeness we mention the limiting behavior in the region {$0<\ga<2, \gamma\leq 2/\ga-1 $,} which has been proved in~\cite{AL11} and~\cite{HM07}. Let $\cP$ be a Poisson Process on $[0,\infty)\times[-1,1]\times[0,1]$ with intensity one. Let $\cP=\{(w_i,x_i, t_i): i\ge 1\}$ be the points in the point process with decreasing values of $w_i$. Let $\cL$ be the set of all real-valued Lipschitz function on $[0,1]$ with Lipschitz constant one and vanishing at $0$. The entropy of a curve $\gc\in\cL$ is $-E(\gc)$ where
\[
E(\gc):=\int_{0}^{1}e(\gc'(x))dx
\]
and $e:[-1,1]\to\dR$ is defined as 
\[
e(x)=\frac12\Big((1+x)\log(1+x)+(1-x)\log(1-x)\Big).
\]
Define the energy of a curve $\gc\in\cL$ as 
\[
\pi(\gc)=\sum_{(w_i,x_{i},t_{i})\in\cP\,:\, \gc(t_{i})=x_{i}} w_{i}^{-1/\ga}.
\]
\
\begin{thm}[\cite{HM07,AL11}]
Assume that the weights satisfy  $\pr(\go>x)=x^{-\ga}L(x)$ for some $\ga\in (0,2)$ and some slowly varying function $L(x)$. Let $\gb_{n}$ be a sequence of real numbers such that $\frac1nm(n^2)\gb_{n} \to \gb$  as $n\to\infty$ for some $\gb\in(0,\infty)\cup\{\infty\}$. Then 
\[
	\frac{1}{m(n^2) \gb_n }\log Z_{n,\gb_n}^{\go}\, \xrightarrow[n\to\infty]{(d)} \,\,\sup_{\gc\in\cL} \Big(\pi(\gc) - \frac{1}{\gb}E(\gc)\Big).
\]
\end{thm}

\subsection{Heuristics}\label{heuristic}

We will now describe the heuristics and the picture (see Figure~\ref{fig1}) that gives the phase diagram for scaling exponents depending on $\gc$ and $\ga$ when $\bar{F}(x)= x^{-\ga+o(1)}, x\gg1$.  We divide the $(\gc,\ga)\in (0,\infty)^2$ plane into regions with different exponents behavior. We define the following regions 
\begin{align*}
\setlength\arraycolsep{0.01em}
\begin{array}{rlrl}
	R_1 &:= \big\{(\gc,\ga): \gc>\frac14, \gc\ge\frac{3}{2\ga} \big\},
	&R_2 &:= \big\{(1/4,\ga): \ga \geq 6 \big\}, \\[2mm]
	R_3 &:=\big\{(\gc,\ga): 0<\gc<1/4, \ga \ge \frac{5-2\gc}{1-\gc}\big\}, 
	&R_4 &:= \big\{(0,\ga): \ga> 5 \big\},\\[2mm]
	R_5 & := \big\{(\gc,\ga): \ga> 1/2,  \max\{0,\frac{2}{\ga}-1, \frac{\ga-5}{\ga -  2}  \} < \gc < \frac{3}{2\ga} \big\}, \\[2mm]
	R_6 &:= \big\{(\gc,\ga):  0< \ga < 2, \gc= \frac{2}{\ga}-1 \big\} 
	\qquad\quad \text{     and }  
	&R_7 &:=  \big\{(\gc,\ga): 0< \ga < 2, 0 \le \gc< \frac{2}{\ga}-1 \big\}.\\[2mm]
\end{array}
\end{align*}

\begin{figure}[hftb]
	\centering                                                        
	\includegraphics[height=70mm]{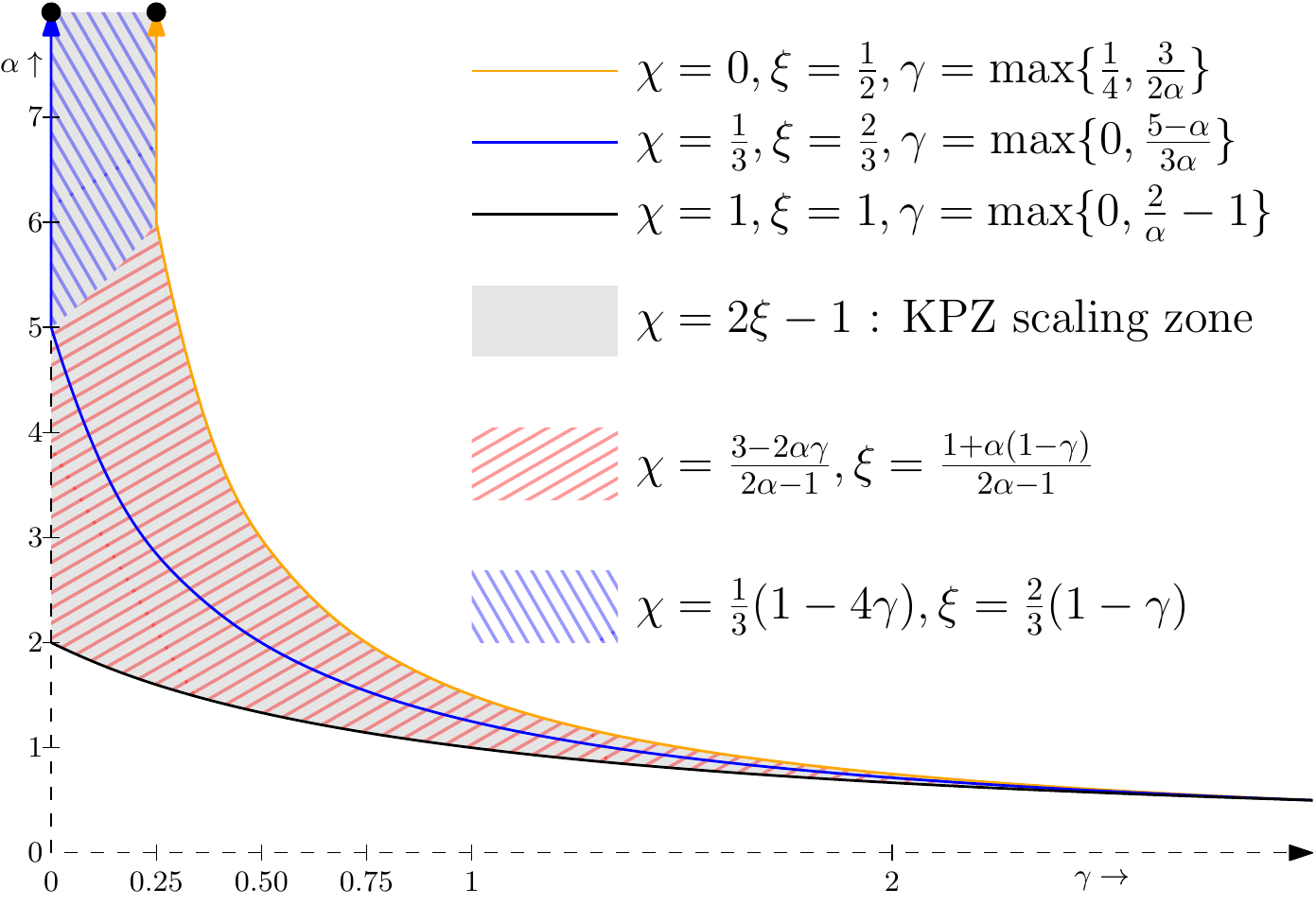}                                                 
	\caption{Phase diagram for fluctuation and transversal exponents} 
	\label{fig1}                                                      
\end{figure}

The region $R_2\cup R_3\cup R_4\cup R_5\cup R_6$ is divided along families of curves characterized by the same exponent $\xi$ (and hence $\chi$). 
More precisely, for any fixed value of $\xi\in[1/2,1]$, the curve in the $(\gc,\ga)$ plane determined by 
\begin{align}\label{eq:level}
\Big\{ (\ga,\gc) \colon \xi= \frac{1+\ga(1-\gc)}{2\ga-1}\,\,\text{and}\,\,\ga\le \frac{5-2\gc}{1-\gc} \Big\}
\,\,\bigcup \,\,\Big\{(\ga,\gc) \colon \xi=\frac{2(1-\gc)}{3}\,\,\text{and}\,\,\ga\ge \frac{5-2\gc}{1-\gc}  \Big\},
\end{align}
should give rise to a polymer measure with transversal-fluctuation-exponent $\xi$ and $\chi=2\xi-1$. In the region $R_1\cup R_7$, the KPZ hyper-scaling relation $\chi=2\xi-1$ fails.
 Inside the region $R_1$ we still have $\xi=\frac12$ but $\chi=\frac14-\gamma<2\xi-1$. Similarly, in region $R_7$ we have $\xi=1$ and $\chi= \frac{2}{\alpha}-\gamma>2\xi-1$. 

\vskip 2mm
In \cite{AKQ10} a heuristic explanation is provided for the exponents $\chi={(1-4\gc)/3}$, $\xi=2(1-\gc)/{3}$ when ``$\ga=\infty$'' (in the sense of finite exponential moments) and $\gc\in(0,1/4)$ based on Airy process heuristics.  
However, this heuristic breaks down when $\ga$ is finite, since the largest weights may play an important role for the exponents, as instead of logarithmic growth the largest weights grow polynomially and can compete with the Brownian fluctuations and entropy terms. To outline our approach, let us fix $\ga>0$, $\gc\ge 0$ and define $\gb_{n}=\gb n^{-\gc}$.  Recall that the energy for the path $\vs$ is $\gb_nH^\go(\vs)$ and the fluctuation exponent $\chi$ gives the typical order of 
	\begin{align*}
		\log Z^{\,\go}_{n,\gb_n} -\E[\log Z^{\,\go}_{n,\gb_n}] = \log \Big\la\exp\Bigl(\gb_n(H^\go - \gl_{n,\gb_{n}})\Bigr)\Big\ra 
	\end{align*}
	as $n^{\chi+o(1)}$, where
	\[
		\gl_{n,\gb_n}:=\frac{1}{\gb_n} \E[\log\la\exp(\gb_n H^\go)\ra],
	\]
	is larger than $0$ by Jensen's inequality (and in fact, typically, it is much larger than $n^{1/2}$). 
Now, take a real number $\gz\in[1/2,1]$ and a box $B_{n,\gz}:=\big( [0,n]\times (-n^\zeta,n^{\zeta})\big)\cap \bZ^2$.  
It is a classic result of order statistics \cite{LLR83} that (here, for simplicity, we ignore slowly varying corrections)
\[
\max\{ \go_v\colon v\in B_{n,\zeta} \} \approx |B_{n,\zeta}|^{1/\ga} \approx n^{\frac{1+\zeta}{\ga}}.
\]
In fact, if $(\go^{(j)}, j=1,2,\ldots,|B_{n,\zeta}|)$ denotes the order statistics of the set $\{ \go_v\colon v\in B_{n,\zeta} \} $ in decreasing order
and $(t^{(j)}, x^{(j)})_{j=1,...,|B_{n,\zeta}|}$ the location inside $B_{n,\zeta}$ where the $j^{\text{th}}$ largest valued is attained, we have that, for any finite $k$
\[
\{n^{-\frac{1+\zeta}{\ga} }\go^{(j)}, \,n^{-1} t^{(j)}, \,n^{-\zeta} x^{(j)}  \}_{j=1,...,k} \,\xrightarrow[n\to\infty]{(d)}\, \{ {\mvw}^{(j)},\, \mvt^{(j)}, \mvx^{(j)} \}_{j=1,...,k},
\]
where are $\mvw^{(j)}$ are non trivial random variables and $\mvt^{(j)}, \mvx^{(j)}$ are independent uniform variables in $[0,1]$.
Hence, the fluctuation of $H$ around any centering will be at least of the order $n^{(1+\gz)/\ga}$. Since the location of the large weights is at scale $n\times n^\zeta$ and uniformly distributed in the box $B_{n,\zeta}$, the entropy cost to catch the large weights will be given by the moderate deviations of the simple random walk (cf. \cite{DZ02})
\begin{equation}\label{moderate}
-\log \vP_n(s_{n\mvt^{(j)}} \approx  n^\zeta \mvx^{(j)}) \approx  n^{2\zeta-1} \,\frac{(\mvx^{(j)})^2}{2\mvt^{(j)}}.
\end{equation}
If we assume that the fluctuations of the partition function are driven by the strategy, which dictates to catch the large weights in a box $B_{n,\zeta}$ (due to a negligible contribution of the rest of the disorder to the energy), then the corresponding energy-entropy competition, will lead to
\begin{align}\label{ene-ent}
n^{\chi+o(1)} \,&\approx \, \log Z^{\,\go}_{n,\gb_n} -\E[\log Z^{\,\go}_{n,\gb_n}] \, \approx  \, \beta_n n^{\frac{1+\zeta}{\ga}} -n^{2\zeta-1}
=n^{\frac{1+\zeta}{\ga}-\gamma} -n^{2\zeta-1}
\end{align}
 The contribution of such strategy will be negligible unless
	\begin{align*}
		2\gz-1 & \le (1+\gz)/\ga-\gc 
		\,\,\text{ \ie\ }  \,\,\gz\le (1+\ga(1-\gc))/(2\ga-1).
	\end{align*}
	Since the largest exponent occurs when equality holds, one expects that the fluctuation exponents will be given by 
\[
\xi= \frac{1+\ga(1-\gc)}{2\ga-1} \quad \text{ and }\quad \chi=2\xi-1=\frac{3-2\ga\gc}{2\ga-1}.
\]
However, when the exponent $\chi=(3-2\ga\gc)/(2\ga-1)$, provided by this strategy, becomes smaller than the fluctuation exponent $\chi=(1-4\gamma)/3$ conjectured for $\ga=\infty$ and same $\gc$, 
we will have a situation where the strategy of catching the large weights is not optimal. In this case, {the strategy, which is intrinsic for $\alpha=\infty$, will prevail leading to} exponents $\chi=(1-4\gamma)/3$ and $\xi=2(1-\gamma)/3$ independent of $\ga$. 
This leads to a decomposition of the $(\ga,\gamma)$ phase diagram determined by the values of $\xi\in[1/2,1]$ and the set of equations~\eqref{eq:level}.
The case $\xi=1/2$, corresponding to diffusive behavior of the polymer, determines that $\gamma=1/4$ for $\ga\geq 6$ (which is consistent with the \cite{AKQ12} conjecture) and $\gamma=3/2\ga$ for $1/2<\ga\le 6$.

	\subsection{Strategy}\label{strategy} 
	We will explain the strategy behind the proof of Theorems \ref{thm:14}, \ref{thm:gauss} and \ref{thm:heavy}. We will mainly focus on the case $\ga> 6$, as it is more transparent.	To formalize the heuristic, one needs to look first at the contribution from what we call the {\it bulk} of the weights to the partition function. 
This is determined by the variables $\{\go_v\}$, which have a value less than a certain cutoff level $k_n$. We will choose
\[
k_n=
\begin{cases}
	\gb_n^{-1}     &, \text{ when } \ga>6,      \\
	\gb_n^{-1}\,\frac{m(n^{3/2}(\log n)^{\eta})}{ m(n^{3/2})} &, \text{  when } \ga\in (\frac12,6] 
\end{cases}
\]
for some $\eta\in(1/2,\ga)$. We then consider the modified environment given by 
	\[
		\tgo_{v}=\go_{v}\ind_{\{\go_v\le k_n\}} \text{ for } v\in \dZ^2.
	\] 
The behavior of $\tgo$ is better compared to $\go$, in the sense that it is bounded and has finite exponential moments. In particular, we have the following result, whose proof is given in Appendix \ref{sec:aux}, 
	\vskip 4mm	
	\begin{lem}\label{lem:comp}
		{\bf I.} Let $\bar{F}(x)=\pr(\go> x)=x^{-\ga}L(x)$, $\ga\geq 2$, and any number $\theta\in(1, \ga)$ such that $\bE[\,\go_+^\theta\,]<\infty$
		and denote $p:=\lfloor \theta \rfloor $. Let also $\lim_{n\to\infty}\gb_n=0$. Then, for any $k_n$ such that $k_n\gb_n$ is bounded away from zero and $\gb_n^{\ga-\theta}e^{\gb_nk_n}\to 0$, as $n$ tends to infinity, we have 
		\begin{align*}
		(i)& \quad \lim_{n\to\infty} \frac{1}{\gb_n^\theta} \Bigl|\E[e^{\gb_n \go_+\ind_{\go\le k_n}}]-1-\sum_{i=1}^{p} \frac{\gb_n^i}{i!} \E[\go_+^i]\Bigr|=0\\
	\text{ and}\quad	(ii)&\quad
			 \lim_{n\to\infty} \frac{1}{\gb_n^\theta} \Bigl|\log \E[e^{\gb_n \go\ind_{\go\le k_n}}]-\log\Big(\E[e^{\gb_n \go_-}]+\sum_{i=1}^{p} \frac{\gb_n^i}{i!} \E[\go_+^i]\Big)\Bigr|=0.
		\end{align*}
\vskip 2mm
               {\bf II.} Assume that  $\bar{F}(x)=\pr(\go> x)=x^{-\ga}L(x)$, with $\ga\in(1/2,2)\setminus \{1\}$. Additionally, assume that $\bE[\go]=0$, if $\ga\in(1,2)$. Then, for any $k_n$ we have
                   \begin{align*}
                      (i)&   \quad   \Bigl|\E[e^{\gb_n \go\ind_{\go\le k_n}}]-1\Bigr| \leq \text{const.} e^{\gb_nk_n} \,\max\{\bar{F}(\gb_n^{-1}),\,\bar{F}(k_n)\}\\
\text{ and} \quad(ii)&\quad  \Bigl|\E[e^{\gb_n \go\ind_{|\go|\le k_n}}]-1\Bigr| \leq C(\gb_n,k_n) \,e^{\gb_nk_n}\max\{\bar{F}(\gb_n^{-1}),\,\bar{F}(k_n)\},
		\end{align*}
		where the constant $C(\gb_n,k_n)$ converges to zero when $\gb_nk_n\to0$.
	\end{lem}
	\vskip 2mm	
	In the above lemma we excluded the case $\ga=1$ purely for exposition purposes since we will not need this case, while the bounds are less explicit due to slowly varying effects.
	
	Lemma \ref{lem:comp} indicates that if we could replace all ``large'' $\go_v$'s by $0$, the new environment, $\tgo_v:=\go_v\ind_{\{\go_v\leq k_n\}}$, behaves in a  much regular way and we can then try to prove a limit theorem for 
	\[
		b_n(\log Z_{n,\gb_n}^{\tgo}-a_n(k_n))
	\]
	for appropriate $a_n(k_n),b_n$.  In the case that $\ga \geq 6$, the desired limit theorem for $\log Z_{n,\gb_n}^\tgo$ falls into the framework of~\cite{AKQ12},~\cite{CSZ15} and a limiting process exists with $b_n=1, k_n=n^{1/4}$ and $a_n(k_n)$ as in Theorem \ref{thm:14}. The next step is to prove that the modified environment gives the dominant contribution in the partition function in the sense that
	\[
		b_n(\log Z_{n,\gb_n}^{\go} - \log Z_{n,\gb_n}^{\tgo})\longrightarrow 0,\,\, \text{ as } n\to\infty
	\]
	in probability for an appropriate choice of the cutoff $k_n$. The strategy to achieve this is as follows: Since we concentrate on the part of the phase space corresponding to exponent $\xi=1/2$, the
main contribution to the partition function should come from paths that stay within  $B=[0,n]\times [-n^{(1+\gd)/2},n^{(1+\gd)/2}]$, for $\gd>0$, chosen appropriately small. Call this set of paths $\mathcal{B}$. Restricting the partition function to the set of such paths we have
\begin{align*}
0\le \log Z_{n,\beta_n}^{\go}(\mathcal{B}) - \log Z_{n,\beta_n}^{\tilde\go}(\mathcal{B}) \le \beta_n \sum_{v\in B} \go_v \ind_{\go_v\ge k_n}.
\end{align*}
The expected value of the right hand side is
\begin{align}\label{energy-in-box}
\begin{split}
\gb_n\E\big[ \sum_{v\in B}\go_v \ind_{\go_v\ge k_n}\big]  
&\le \gb_n |B| \E[\go \ind_{\go\ge k_n}] \\
&\le \text{const.} \,\gb_n n^{\frac{3+\gd}{2}} k_n^{1-\ga} \E[\go_+^\ga]=O(n^{\frac{6-\ga}{4}+\frac{\gd}{2}}),
\end{split}
\end{align}
where we made use of the choice $k_n=O(n^{1/4})$ and, for the sake of exposition, we made the slightly stronger assumption that $\E[\go_+^\ga]<\infty$. The assumption that $\ga>6$ and a choice of appropriately small $\gd$ will show that the partition functions of the original and the truncated variables, when the
 paths are restricted to box $B$ is asymptotically zero, in probability. The contributions to the partition function from paths that exit the box $B$ will be controlled by 
 an energy-entropy estimate in the spirit of \eqref{ene-ent}. However, to succeed in this control, a multi-scale argument is needed, which will consider the energy-entropy balance in a sequence of boxes $B_j:=[0,n]\times[-n^{(1+\delta_j)/2}, n^{(1+\delta_j)/2} ]$, for appropriate choice of increasing positive numbers $(\gd_j)_{j\geq 1}$. The basic principle is that the energy collected by paths staying inside box $B_j$ (but exiting box $B_{j-1}$) is bounded by $\gb_n\E\big[ \sum_{v\in B_j}\go_v \ind_{\go_v\ge k_n}\big] = O\big(n^{\frac{6-\ga}{4}+ {\gd_j}/{2}}\big)$, as in \eqref{energy-in-box}. However, the entropy cost to exit the box $B_{j-1}$ (in order to catch the large weights outside of it) is of the order
 $O(n^{\gd_{j-1}} )$, by \eqref{moderate}. For $\ga\ge6$ the entropy will dominate the energy, as long as $\gd_j/2<\gd_{j-1}$ and this relation can be iterated, to cover all the scales.
 \vskip 2mm
Similar strategy is applied in the case $\frac12<\ga\le6$. However, the cutoff has to be chosen differently $k_n=m(n^{3/2}(\log n)^{\eta})/m(n^{3/2})$, with $\eta\in(1/2,\ga)$, leading to a more subtle multi-scale procedure with a logarithmic number of iterations, instead of a finite that was sufficient in the $\ga>6$ case. 
 \vskip 2mm
Let us also give a brief explanation of the origins of the different scaling limits obtained in the three theorems. As we already mentioned, once the truncation is achieved, the limit behavior coincides with that of the truncated partition function $Z^\tgo_{n,\gb_n}$. In the case $\alpha>2$ the necessary centering will correspond to normalizing the partition function by the moment generating function $\gl_n(\gb_n)$ of the truncated variables $\tgo=\go \ind_{\{\go\leq k_n\}}$. Denoting by $\gz^{(n)}_{i,x}:=  e^{ \gb_n\tgo_{i,x} -\gl_n(\gb_n)}-1$ and performing a binomial expansion we write 
\begin{align*}
e^{-n\gl_n(\gb_n)} Z_{n,\gb_n}^\tgo 
&= \Big\la e^{\sum_{i=1}^n (\gb_n\tgo_{i,s_i} -\gl_n(\gb_n))}\Big\ra \\
&=  \Big\la \prod_{i=1}^n\big( 1+\gz^{(n)}_{i,s_i})  \Big\ra 
=1 +\sum_{1\leq i\leq n,\,x\in\bZ} \gz^{(n)}_{i,x} p_n(i,x) \\
&\qquad\qquad\qquad\qquad+\sum_{\substack{1\leq i_1<i_2  \leq n \\ x_1,x_2 \in \bZ}} \gz^{(n)}_{i_1,x_1} \gz^{(n)}_{i_2,x_2} \,\,p_n(i_1,x_1) p_n(i_2,x_2)+\cdots
\end{align*}
Thanks to the centering $\gl_n(\gb_n)$ the random variables $\gz^{(n)}_{i,x}$ have mean zero. Computing the variance of the first nontrivial term in the above expansion, we have that
\begin{align*}
\bV\text{ar} \big( \sum_{1\leq i\leq n,\,x\in\bZ} \gz^{(n)}_{i,x} p_n(i,x)  \big) = \big( e^{\gl_n(2\gb_n)-2\gl_n(\gb_n)}-1 \big) \sum_{1\leq i\leq n,\,x\in\bZ} (p_n(i,x))^2.
\end{align*} 
 Lemma \eqref{lem:comp} shows that when $\ga>2$ then $e^{\gl_n(2\gb_n)-2\gl_n(\gb_n)}-1\approx \gb_n^2$. Using the local limit theorem  i.e. $\sqrt{n} p_n(i,x)\approx 2\rho(i/n,x/\sqrt{n})$, (where we have also taken into account the periodicity of the  simple random walk), we finally see that  
 \begin{align*}
\bV\text{ar} \big( \sum_{1\leq i\leq n,\,x\in\bZ} \gz^{(n)}_{i,x} p_n(i,x)  \big)\approx 2\gb_n^2 \sqrt{n} =(\sqrt{2}\gb_n n^{1/4})^2.
\end{align*} 
In the case $\ga>2$ the random variables $\tilde\gz^{(n)}_{i,x}:=\gb_n^{-1} \gz^{(n)}_{i,x}$ will have uniformly integrable second moments and we will be in a regime where the multilinear polynomials, below, will have a well defined limit, i.e. 
\begin{align*}
e^{-n\gl_n(\gb_n)} Z_{n,\gb_n}^\tgo
&=1 +\frac{(\gb_n n^{1/4})}{n^{3/4}}\sum_{1\leq i\leq n,\,x\in\bZ} \tilde\gz^{(n)}_{i,x} \,\big(\sqrt{n}p_n(i,x)\big)\\
&\quad\quad+ \frac{(\gb_n n^{1/4})^2}{(n^{3/4})^2}\sum_{\substack{1\leq i_1<i_2  \leq n \\ x_1,x_2 \in \bZ}} \tilde\gz^{(n)}_{i_1,x_1} \tilde\gz^{(n)}_{i_2,x_2} \,\,\big(\sqrt{n}p_n(i_1,x_1)\big) \big(\sqrt{n}p_n(i_2-i_1,x_2-x_1)\big)+\cdots\\
&\approx 1 +(\sqrt{2}\gb_n n^{1/4}) \int_0^1\int_\bR \gr(t,x) W(\dd t\, \dd x ) \\
&\qquad\qquad+
 (\sqrt{2}\gb_n n^{1/4})^2 \int_{0<t_1<t_2<1}\int_{\bR^2} \prod_{i=1,2}\gr(t_i-t_{i-1},x_{i}-x_{i-1}) W(\dd t_{i}\, \dd x_i )+\cdots
\end{align*} 
When $\gb_n n^{1/4}\to\gb\in (0,\infty)$, we will be in the regime covered by Theorem \ref{thm:14}. When $\gb_n n^{1/4}\to 0$, which is the case of Theorem \ref{thm:gauss}, the terms of order two or higher will be negligible compared to the first two terms, while the distribution of the linear term is asymptotically Gaussian, i.e. 
\begin{align}\label{expansion}
e^{-n\gl_n(\gb_n)} Z_{n,\gb_n}^\tgo&\approx 
1 +(\sqrt{2}\gb_n n^{1/4})\int_0^1\int_\bR \gr(t,x) W(\dd t\, \dd x ) +o(\gb_n n^{1/4})
\end{align} 
\vskip 2mm
In the case $\ga<2$, the random variables $\tilde\gz^{(n)}_{i,x}:=\gb_n^{-1} \gz^{(n)}_{i,x}$ fail to have uniformly integrable second moments (or even a second moment), due to the necessary choice of the truncation level $k_n\gg \gb_n^{-1}$ and so we come out of a central limit theorem regime. The largest weights will now have a dominant role and the choice of $\gb_n=\gb \cdot m(n^{3/2})^{-1}$
will create an asymptotically Poissonian field $\cP(\dd w \dd x \dd t)$ for the values $\{(\gb_n \go_{i,x},x,i) \,:\, 1\leq i\leq n, |x|=O(\sqrt{n}) \}$. 
Let us write a multilinear expansion of the partition function $Z_{n,\gb_n}^\go$ (without a truncation, although we will eventually need to make a suitable truncation)
\begin{align*}
 Z_{n,\gb_n}^\go
 &= 1 +\sum_{1\leq i\leq n,\,x\in\bZ} (e^{\gb_n\go_{i,x}}-1) p_n(i,x) \\
&\qquad\qquad+
\sum_{\substack{1\leq i_1<i_2  \leq n \\ x_1,x_2 \in \bZ}} 
\prod_{j=1,2} (e^{\gb_n\go_{i_j,x_j}}-1)  \,\,p_n(i_j-i_{j-1},x_j-x_{j-1}) +\cdots\\
&\approx 1+ \frac{1}{\sqrt{n}} \sum_{1\leq i\leq n,\,x\in\bZ} (e^{\gb_n\go_{i,x}}-1) \,(\sqrt{n}p_n(i,x)) + O({1}/{n})\\
&``\approx "\, 1+ \frac{1}{\sqrt{n}} \int (e^{\gb w}-1) \rho(t,x) \cP(\dd w\dd x \dd t) + O({1}/{n}),
\end{align*}
where we put the last approximation in quotations, since a careful centering will need to be made, in order to guarantee the well definiteness of the Poissonian integrals. This will also alter the final formulation of the last line.

The limiting partition function will be a weighted summation over the Poissonian field of $e^{\gb w}-1$ weighted by the Gaussian probability $\rho(x,t)$. This is to be interpreted that the polymer will search for all the sites with a large weight that will give energy $e^{\gb w}-1$ and an entropy $n^{-1/2}\rho(x,t)$ will have to be paid. 
However, it will be too costly for the polymer to get simultaneous contribution from two such sites, as the entropy that will be paid will be of order $(1/\sqrt{n})^{\sharp\{\text{sites visited}\}}$. Finally, it is also worth comparing the contribution to the scaling limit from a Gaussian behavior, driven by the bulk of ``small'' weights. By \eqref{expansion} this will be of order $\gb_n n^{1/4} =m(n^{3/2})^{-1} n^{1/4}\approx n^{(\ga-6)/(4\ga)}\ll n^{-1/2}$, for $\ga<2$. Therefore, it is also seen in this way that the Gaussian fluctuations, driven by the bulk of small weights, is negligible compared to the Poissonian fluctuations, driven by the few large ones.

\subsection{Conjectured Behavior in Region $R_5$}
Let us close by discussing a conjectured behaviour for $(\ga,\gc)$ in the region 
\[
R_5:= \{(\gc,\ga): \ga> 1/2,  \max\{0,2/\ga-1, (\ga-5)/(\ga - 2) \} < \gc < 3/2\ga\}
\]
 and $\go$ satisfying~\eqref{defA}. The heuristic idea presented in Section~\ref{heuristic} shows that weights outside a box of width much larger
  $ h_{n}=n^\xi$ with $\xi=\big(1+\ga(1-\gamma)\big)/(2\ga-1)$
 should not contribute much, whereas the big weights in the box of width $h_n$ should give the dominant contribution. The temperature is scaled 
 as $\gb_n=\hat\beta n^{-\gamma}$, with $\gamma=\ga^{-1}(1+\xi)-\chi$ and $\chi=2\xi-1$ and we are asking for a distributional limit for
\[
n^{-\chi}\big(\,\log Z_{n,\gb_n}^{\go} - a_n\,\big)
\]
 for some centering $a_n$. The energy $H^\go(\vs)$ of a path $\vs$ restricted to a box of width $h_n$ will be
 \[
 \gb_n H^\go(\vs):=n^\chi \sum_{v\in[0,n]\times (-h_n,h_n),\, v\in \vs} n^{-\frac{1+\xi}{\ga}}\,\go_{v} \approx 
 n^{\chi} \sum_{\cP} u_i,
 \]
 where $\cP=\{(x_i, t_i, u_i): i\ge 1\}$ is the limiting Poisson point process consisting of the locations $(x_i,t_i)$ and sizes $u_i$ of the scaled variables 
 $n^{-\frac{1+\xi}{\ga}}\,\go_{v}$ inside the box of width $h_n$. Moreover, an entropy cost will have to be paid for the path to catch the big, scaled
 weights $u_i$. The size of the cost to go from one point $(x_i,t_i)\in \cP$ to another one $(x_j,t_j)\in \cP$
  will be $n^\chi |x_i-x_j|^2/(2|t_i-t_j|)$, cf. \eqref{moderate}, which matches the size of the energy fluctuations. Therefore, one can expect that when the
   logarithm of the partition function is scaled by $n^\chi$ and is centered appropriately, it will converge in distribution to a random 
   variational problem described by the random  variable
\begin{align*}
	T_{\hat\beta} :=\sup_{A\subset \cP \colon |A|<\infty}\sum_{i\in A} \Big(\hat\beta u_i - \frac{\, \, |x_i-x_{i-1}|^2}{2|t_i-t_{i-1}|} \Big) 
\end{align*}
We conjecture that:
\begin{conj}
For $\ga\in (1/2,2]$, $T$ is well-defined. Moreover, for $(\ga,\gc)\in R_{5}, \ga<2$, $\gb_n=\hat\beta n^{-\gamma}$, with $\gamma=\ga^{-1}(1+\xi)-\chi$ and $\chi=2\xi-1$ and $\go$ satisfying~\eqref{defA}, there exist constants $a_{n}$  such that
\[
n^{-\chi}(\log Z_{n,\gb_n}^{\go} - a_n) \xrightarrow[]{(d)} T_{\hat\beta},\quad \text{ as } n\to\infty.
\]
\end{conj}
One difficulty to establish this is the identification of the centering $a_n$, which would cancel the contributions from moderate size disorder weights.
Since in region $R_5$ it holds that $\xi>1/2$, we are not any more in the weak disorder regime and one would expect that $a_n$ is not any more 
identical to a truncated log-moment generating function.


 \subsection{Roadmap} 
We perform the truncation via the above mentioned multi-scale argument in Section~\ref{sec:trunc}. 
In Section~\ref{sec:conv}, we show that the partition function with truncated disorder converges to the desired limits when $\ga>2$. 
Finally, we identify the scaling limit in the case of $\ga< 2$ in Section~\ref{sec:le2}. In the appendix we provide the proofs of some auxiliary estimates.
Regarding notation, we will be writing $const.$ for a generic constant that does not depend on specific parameters. We will use freely the symbols $o(\cdot), O(\cdot)$ and when we 
want to put emphasis on the parameters, we will add these as subscripts, e.g. we will write $o_n(1)$, if we want to emphasize that a quantity converges to zero when $n$ tends to infinity. We will often interchange freely between the notation $v$ and $(i,x)$
for points in $\bN\!\times\!\bZ$.  
\section{Comparing the original and truncated partition functions}\label{sec:trunc}

Recall that the original environment is given by $\gO=\{\go_{v}:v\in\dZ^{2}\}$ and the truncated environment by $\tilde\gO=\{\tgo_{v}:v\in\dZ^{2}\}$ where $\tgo=\go\ind_{\{\go\le k_n\}}$. 
To show that the difference $\log Z^{\go}_{n,\gb_n} - \log Z^{\tgo}_{n,\gb_{n}}$ 
is ``small'', we use the multi-scale argument, outlined in the heuristics. Let us set up the framework introducing some notation. 
\vskip 2mm
Given a sequence of integers $0=h_0<  h_1 < h_2<\cdots <h_\ell$ such that $h_{\ell-1}<n\le h_{\ell}$, for some $\ell\ge 1$, we define the corresponding cylinder blocks as
\begin{align}\label{eq:block}
	\vB_j=[0,n]\times (-h_j, h_j) \text{ for } j=1,2,\ldots,\ell 
\end{align}
and the set of paths restricted to $\vB_{j}$ as
\begin{align}\label{eq:pblock}
	\cB_j=\{((i,s_i))_{i=0}^ n\in\sS_0^n: \max_{1\le i\le n} |s_i|< h_j\},\,\, \text{ for } \,\,i=1,2,\ldots,\ell. 
\end{align}
In other words, $\cB_j$ is the collection of random walk paths contained in the cylinder $\vB_j$ and the set $\cB_j\setminus \cB_{j-1}$ is the set of paths that exit the cylinder $\vB_{j-1}$ by time $n$ but not the cylinder $\vB_j$. Here $\ell, \{h_{i}, i\ge 1\}$ will depend on $n$ and the tail behavior of $\go$. We define $\vB_{0}=\cB_{0}=\emptyset$.  

We also need to recall the well known fact, that the probability of the set $\cB_j^c$ under the simple random walk path measure is bounded by $4\exp(-h_j^2/2n)$ :
\begin{lem}[cf.~Feller (1968)]\label{Feller} 
	Under the simple random walk measure $\vP_{n}$ and for any positive integer $r$ we have
	\[
		\vP_n\big(\max_{0\le i\le n} s_i \ge r\big) = 2\vP_n(s_n\ge r) - \vP_n(s_n=r).
	\]
	Thus,
	\[
		\vP_n\big(\max_{0\le i\le n} |s_i| \ge r\big) \le  2\vP_n(s_n\ge r) \le 4 e^{-r^2/2n}.
	\]
\end{lem}
\vskip 4mm
Let us now define $Z^{\go}_{n,\gb_n}(\cA):= 2^{-n}\sum_{\vs\in \cA} \exp(\gb_n H^{\go}(\vs)) \text{ for any } \cA\subseteq \sS_{0}^n$ so that
\begin{equation}\label{deco}
	Z^{\go}_{n,\gb_n} = \sum_{j=1}^{\ell} Z^{\go}_{n,\gb_n}(\cB_{j}\setminus \cB_{j-1}) \,\,\text{ and }\,\,
	Z^{\tgo}_{n,\gb_n} = \sum_{j=1}^{\ell} Z^{\tgo}_{n,\gb_n}(\cB_{j}\setminus \cB_{j-1}).
\end{equation}
For every $j=1,2,\ldots,\ell$, we define 
\begin{align}\label{eq:excess}
	M_j :=\sum_{v\in \vB_j} (\go_v - \tgo_v) = \sum_{v\in \vB_j} \go_v \ind_{\{\go_v>k_n\}} ,
\end{align}
as the total excess weight in block $\vB_j$. Note that for any path $\vs=((i,s_{i}))_{i=0}^{n}\in \cB_j$ we have $0\le H^{\go}(\vs) -  H^{\tgo}(\vs) \le M_j$ and thus
\begin{equation}\label{deco-est}
	Z^{\go}_{n,\gb_n}(\cB_{j}\setminus \cB_{j-1})  \le e^{\gb_n M_j} Z^{\tgo}_{n,\gb_n}(\cB_{j}\setminus \cB_{j-1}) \text{ for all } j=1,2,\ldots,\ell.
\end{equation}
Combining  \eqref{deco}, \eqref{eq:excess} and \eqref{deco-est} we have

\begin{lem}\label{lem:compare}
	For any real number $A>0$ we have
	\begin{align*}
		  & \pr(\log Z_{n,\gb_n}^\go  - \log  Z_{n,\gb_n}^{\tgo}>A) \\
		  &\qquad\qquad
		\le 
		|\vB_1|\cdot\bar{F}(k_n)
		+ 
		\pr\Big(\sum_{j=2}^{\ell}\exp( \gb_n M_j)\, \vP_{n,\gb_n}^{\tgo}(\cB_{j-1}^c) > A \Big).
	\end{align*}
\end{lem}
\begin{proof}
Using decomposition \eqref{deco} and estimate \eqref{deco-est}, in the second inequality below, we have
\begin{align*}
	1\le \frac{ Z_{n,\gb_n}^\go }{ Z_{n,\gb_n}^{\tgo} } 
	  & = 1 + \frac{ Z_{n,\gb_n}^\go - Z_{n,\gb_n}^\tgo }{ Z_{n,\gb_n}^\tgo }                                                                                                                \\
	  & \le 1 + \frac{ Z^\go_{n,\gb_n}(\cB_1) - Z^{\tgo}_{n,\gb_n}(\cB_1) }{ Z^\tgo_{n,\gb_n} } + \sum_{j=2}^{\ell}  (e^{\gb_n M_j}-1)\cdot \vP_{n,\gb_n}^{\tgo}(\cB_{j}\setminus \cB_{j-1}) \\
	  & \le 1 + \frac{ Z^\go_{n,\gb_n}(\cB_1) - Z^{\tgo}_{n,\gb_n}(\cB_1) }{ Z^\tgo_{n,\gb_n} } + \sum_{j=2}^{\ell} e^{\gb_n M_j}\cdot \vP_{n,\gb_n}^{\tgo}(\cB_{j}\setminus \cB_{j-1}).     
\end{align*}
Using $\log(1+ x)\le x$, this immediately implies that 
\begin{align}\label{eq:compare}
	0\le \log Z_{n,\gb_n}^\go  - \log  Z_{n,\gb_n}^{\tgo} \le \frac{ Z^\go_{n,\gb_n}(\cB_1) - Z^{\tgo}_{n,\gb_n}(\cB_1) }{ Z^\tgo_{n,\gb_n} } + \sum_{j=2}^{\ell} e^{\gb_n M_j}\cdot \vP_{n,\gb_n}^{\tgo}(\cB_{j}\setminus \cB_{j-1}). 
\end{align}
	
Note that on the event 
$
\{\max_{v\in \vB_1} \go_v \le k_n \}
$
we have 
$
Z^{\go}_{n,\gb_n}(\cB_1)  = Z^{\tgo}_{n,\gb_n}(\cB_1).
$ 
Moreover, we have 
$\pr(\max_{v\in \vB_1} \go_v > k_n) \le |\vB_1|\cdot\bar{F}(k_n)$ and our claim follows.
\end{proof}
\vskip 4mm
Now,
$
\E[\go\ind_{\{\go>k_n\}}] =k_n^{1-\ga}L_{1}(k_n)
$
with the slowly varying function 
\[
	L_1(t):= L(t) + \int_1^{\infty}x^{-\ga}L(tx)\,\dd x.
\]
For $\ga\ne1$, it is easy to check that there exist constants $c>1, t_0>1$, such that
\begin{equation}\label{slowcomp}
	L(t)\le L_{1}(t)\le cL(t) ,\,\,\text{ for all }\, t\ge t_0. 
\end{equation}
Clearly $\E[M_j]=\E[\sum_{v\in\vB_j} \go\ind_{\{\go> k_n\}}]=|\vB_j| k_n^{1-\ga}L_1(k_n)=nh_j\,k_n^{1-\ga}L_1(k_n)$.  Thus, if 
\[
	\gb_n\,nh_j\, k_n^{1-\ga} \ll h_{j-1}^2/n,
\]
then the expected energy accumulated by the path (this corresponds to the left hand side of the above inequality) will be dominated by the entropy cost (that corresponds to the right hand side of the above inequality). Hence, the contribution of the set of paths in $\cB_{j}\setminus \cB_{j-1}$ will be small. The following Proposition~\ref{prop:bound} makes this argument rigorous. Recall that the function $m:(1,\infty)\mapsto\dR$ from \eqref{mt}, which satisfies $t\bar{F}(m(t))=1$.

\begin{prop}\label{prop:bound}
	Consider a sequence $(\gb_n)_{n\ge 1}$ such that $\sup_{n\ge 1}\gb_n \max\{n^{1/4},m(n^{3/2})\}<\infty$. Assume that $\ga>1/2$ and let 
	\[
		k_n=
		\begin{cases}
			\gb_n^{-1}                                                   &, \text{ when } \ga>6,        \\
			\gb_n^{-1}\,\frac{m(n^{3/2}(\log n)^{\eta})}{ m(n^{3/2})} &, \text{  when } \ga\in (\frac12,6] 
		\end{cases}
	\]
	for some $\eta\in(1/2,\ga)$. Assume, also, that 
	\begin{equation}\label{posi-est}
		\lim_{\eps\,\downarrow \,0}\limsup_{n\to\infty}\pr\left(Z_{n,\gb_n}^{\tgo}< \eps \E[Z_{n,\gb_n}^{\tgo}]\right) = 0.
	\end{equation}
	Then, for any $ a\ge 0$
	\[
		n^{a}\big(\log Z_{n,\gb_n}^\go  - \log  Z_{n,\gb_n}^{\tgo}\big)\xrightarrow[n\to\infty]{P} 0
	\]
\end{prop}
\begin{proof}
Recall the definitions of $\vB_{j},\cB_{j},M_{j},j=1,2,\ldots,\ell$, from equations~\eqref{eq:block}, \eqref{eq:pblock} and  \eqref{eq:excess} for a given sequence of heights $\sqrt{n}\ll h_1 < h_2<\cdots <h_{\ell} $ with $h_{\ell-1}<n\le h_{\ell}$, that will be determined later on. We distinguish between three cases:
\vskip 2mm

\noindent{\bf Case 1.}\label{case1} ($\ga>8$) First, we consider the case $\ga>8$ and choose $k_n\ge n^{1/4}$. If we take $\ell=1$ with $h_{1}=n$, we have by Lemma \ref{lem:compare}, {with $A=n^{-a}$ and arbitrary $a>0$} (notice that in this case the second term in the inequality does not appear) that 
\begin{align*}
 \pr(\log Z_{n,\gb_n}^\go  - \log  Z_{n,\gb_n}^{\tgo}> n^{-a}) 
	&\le |\vB_1|\cdot\bar{F}(k_n) \le |\vB_1|\bar{F}(n^{1/4})\\
        &= n(2n+1) \,n^{-\ga/4}L(n^{1/4})\longrightarrow 0,\,\, \text{ as }\,\, n\to \infty,
\end{align*}
which proves the claim in this case.
 \vskip 2mm
Let us now work towards the case $\ga\le 8$. To prepare, we start by using Lemma \ref{lem:compare} with the value $A$, therein, chosen again to be $n^{-a}$.
Choosing an arbitrary number $\eps>0$ and making elementary probability estimates, we have
\begin{align}\label{boundes0}
\begin{split}
 \pr(\log Z_{n,\gb_n}^\go  - \log  Z_{n,\gb_n}^{\tgo}> n^{-a})
& \le |\vB_1|\cdot\bar{F}(k_n)
+ 
\pr\Big(\sum_{j=2}^{\ell}\exp( \gb_n M_j) \vP_{n,\gb_n}^{\tgo}(\cB_{j-1}^c) > n^{-a} \Big)\\
&\le  |\vB_1|\cdot\bar{F}(k_n)
+\pr\Big(Z_{n,\gb_n}^{\tgo}< \eps \E[Z_{n,\gb_n}^{\tgo}]\Big) \\
&\qquad\qquad+ \sum_{j=2}^{\ell}\pr\Big( \gb_n M_j +\log \frac{Z_{n,\gb_n}^{\tgo}(\cB_{j-1}^c)}{\E[Z_{n,\gb_n}^{\tgo}]}  > \log(\eps \,\ell^{-1}n^{-a})\Big).
\end{split}
\end{align}
Furthermore,  we estimate
\begin{align}\label{boundes1}
&\pr\Big( \gb_n M_j +\log \frac{Z_{n,\gb_n}^{\tgo}(\cB_{j-1}^c)}{\E[Z_{n,\gb_n}^{\tgo}]}  > \log(\eps\, \ell^{-1} n^{-a})\Big) \nonumber\\
	  & \qquad\le \pr\Big(Z_{n,\gb_n}^{\tgo}( \cB_{j-1}^c) \ge \eps^{-1}\ell\,\,\vP(\cB_{j-1}^c)\E[Z_{n,\gb_n}^{\tgo}]\Big)           
	+\pr\Big( \gb_n M_j  > \log\big(\,(\eps \ell^{-1})^2 \,n^{-a}\big) - \log \vP(\cB_{j-1}^c) \Big).
\end{align}
Assuming that $ \log\big(\,(\eps \ell^{-1})^2 \,n^{-a}\big) - \log \vP(\cB_{j-1}^c)>0$ (this assumption will be satisfied by the choices of the parameters $\eps,\ell,h_j$), we use Chebyshev's inequality in both terms of \eqref{boundes1} and the fact that
\[
	\E[Z_{n,\gb_n}^{\tgo}(\cA)] = \vP(\cA) \E[Z_{n,\gb_n}^{\tgo}], \quad \text{ for any } \cA\subseteq \sS_{0}^{n}.
\]
to estimate  \eqref{boundes1} as
\begin{align}\label{boundes2}	
           \pr\Big( \gb_n M_j +\log \frac{Z_{n,\gb_n}^{\tgo}(\cB_{j-1}^c)}{\E[Z_{n,\gb_n}^{\tgo}]}  > \log(\eps\, \ell^{-1} n^{-a}) \,\Big) 
          &\le \frac{\eps}{\ell} + \frac{\gb_n  \E[M_j] }{ \log\big(\,(\eps \ell^{-1})^2 \,n^{-a}\big) - \log \vP( \cB_{j-1}^c)}\nonumber\\
	  & \le \frac{\eps}{\ell}  + \frac{\gb_n |\vB_j|\E[{\go\ind_{\{\go>k_n\}} }]}{ \log\big(\,(\eps \ell^{-1})^2 \,n^{-a}\big) - \log \vP( \cB_{j-1}^c)}\nonumber\\
	  & \le \frac{\eps}{\ell} +  \frac{\gb_n |\vB_j|\E[{\go\ind_{\{\go>k_n\}} }]}{ \log\big(\,(\eps \ell^{-1})^2 \,n^{-a}/4\big) +h_{j-1}^2/2n}
\end{align}
where the last inequality follows by the result that
$
\vP(\cB_{j}^c)\le 4\exp(-h_j^2/2n),
$
for all $j\ge 1$, by Lemma \ref{Feller}.
Combining \eqref{boundes2} and \eqref{boundes0} we have
\begin{align}
	  & \pr(\log Z_{n,\gb_n}^\go  - \log  Z_{n,\gb_n}^{\tgo}> n^{-a}) \notag \\
	  & \quad\le                                                                     
	|\vB_1|\,\bar{F}(k_n)
	+
	\pr\left(Z_{n,\gb_n}^{\tgo}< \eps \E[Z_{n,\gb_n}^{\tgo}]\right) 
	+ \,
	\eps\, + \sum_{j=2}^{\ell}  \frac{\gb_n |\vB_j|\E[{\go\ind_{\{\go>k_n\}} }]}{ \log\big(\,(\eps \ell^{-1})^2 \,n^{-a}/4\big) +h_{j-1}^2/2n},
	\label{eq:pbd1}
\end{align}
We are now ready for
\vskip 2mm

\noindent{\bf Case 2.} ($6<\ga\le 8$)  Here  the cutoff is $k_{n}=\gb_n^{-1}$. Note that if $\gb_n\ll n^{-2/\ga}$ then as in Case~1, we can take $h_{1}:=n$ with $\ell=1$ to get the result and avoid the multi-scale argument. Thus, w.l.o.g.~we can assume that $\gb_{n}\gg n^{-2/\ga}$. From the assumption $\sup_{n\ge 1}\gb_n \max\{n^{1/4},m(n^{3/2})\}<\infty$, we have $k_n\ge \text{const.} \,n^{1/4}$ and we choose
$$h_{j}=\lfloor n^{(1+\gd_j)/2}\rfloor,\qquad \text{for}\,\, j\ge 1,$$
with
\begin{align*}
	\gd_{1}&=\frac{1}{4}(\ga-6),\,\, \qquad\qquad\qquad\qquad\qquad\quad \text{and} \notag\\
	\gd_{j}&=2\gd_{j-1}+\frac{\ga-6}{4}=\frac{2^j-1}{4}(\ga-6),\,\,\,\quad\text{ for } j=2,3,\ldots,\ell  ,
\end{align*}
and $\ell=\lceil\log_{2}(1+4/(\ga-6))\rceil$, so that $\gd_{\ell}\ge 1$, which guarantees $h_\ell\geq n$.
We notice that, for all large enough $n$, we have
\begin{align}\label{pbd1denom}
 \log\big(\,(\eps \ell^{-1})^2 \,n^{-a}/4\big) +h_{j-1}^2/2n \geq h_{j-1}^2/4n.
\end{align}
 Moreover, we have that
 \begin{align}\label{pbd11}
  |\vB_1|\bar{F}(k_{n}) &\le \text{const}. \,|\vB_1|\bar{F}(n^{1/4})\le \text{const}.\, nh_{1}\cdot\,n^{-\ga/4}L(n^{1/4})\notag\\
  &\le\text{const}.\, n^{\frac{6-\alpha}{8}} L(n^{1/4}) =o(1) ,
  \end{align}
  since $\alpha>6$.
Using the choice $k_n=\gb_n^{-1}$ and the assumption that $\gb_{n}n^{1/4}$ is bounded we also have 
\begin{align}
	\gb_n |\vB_j|\E[\go\ind_{\{\go>k_n\}}] &=   \gb_n\,(2nh_j)\, k_n^{1-\ga}L_1(k_n)=
	\,(2nh_j)\, \gb_n^{\ga}L_1(k_n)\notag\\
	&\le\text{const.} n^{1+\frac{1+\gd_j}{2}-\frac{\ga}{4}} L_1(k_n)
	= \text{const.} \,n^{\frac{\gd_j}{2}+\frac{6-\ga}{4}} L_1(n^{1/4})\notag\\
	&= \text{const.} \,n^{\gd_{j-1}+\frac{6-\ga}{8}} L_1(n^{1/4})\notag\\
	&= \text{const.} \,n^{-1} h_{j-1}^2 n^{\frac{6-\ga}{8}} L_1(n^{1/4}).\label{pbd13}
\end{align}
 Inserting the bounds \eqref{pbd1denom}, \eqref{pbd11},\eqref{pbd13} into \eqref{eq:pbd1} we obtain
 \begin{align*}
	&\pr(\log Z_{n,\gb_n}^\go  - \log  Z_{n,\gb_n}^{\tgo}> n^{-a})  \\
	&   \qquad   \qquad
	\le \text{const.}n^{\frac{6-\ga}{8}}L(n^{1/4})+ \pr\left(Z_{n,\gb_n}^{\tgo}< \eps \E[Z_{n,\gb_n}^{\tgo}]\right) + \eps +\text{const.} n^{\frac{6-\ga}{8}}L_1(n^{1/4})
\end{align*}
and so
\begin{align*}
	\limsup_{n\to\infty}\pr(\log Z_{n,\gb_n}^\go  - \log  Z_{n,\gb_n}^{\tgo}> n^{-a})        
	\le \limsup_{n\to\infty} \pr\left(Z_{n,\gb_n}^{\tgo}< \eps \E[Z_{n,\gb_n}^{\tgo}]\right) + \eps. 
\end{align*}
Taking $\eps\downarrow0$ we have the result for $\ga\in (6,8]$, {by the assumption that in this limit the first term on the right hand side vanishes}. \\

\noindent{\bf Case 3.} ($1/2<\ga\le 6$) Now we consider the case when $1/2<\ga\le 6$. We choose
\[
	k_n=\frac{m(n^{3/2}(\log n)^{\eta})}{\gb_n m(n^{3/2})} \ge \text{const.} m(n^{3/2}(\log n)^{\eta}) ,
\]
for some $\eta\in(1/2,\ga)$.  As before, we can assume that $\gb_n \gg n^{-2/\ga}$, otherwise the proof is trivial. 
We now choose 
 $h_{j}=\lfloor\sqrt{n}(\log n)^{\gd_j}\rfloor$ for $j\ge 1$, with
\begin{align*}
	\gd_{j}&:=\frac12+\frac{(2^j-1)}{4}(2\eta-1),\,\qquad\qquad \text{for}\,\, j=1,2,\ldots,\ell,
\end{align*}
and $\ell=\ell_n:=\lceil\log_2(1+ \frac{2\log n}{(2\eta-1)\log\log n}) \rceil$, so that $h_{\ell}\ge n$. 
We first notice that this choice of $h_j$ implies, similarly to \eqref{pbd1denom}, that
\begin{align}\label{pbd1denom2}
\log\big(\,(\eps \ell^{-1})^2 \,n^{-a}/4\big) +h_{j-1}^2/2n 
&\ge \log\big(\,(\eps \ell^{-1})^2 \,n^{-a}/4\big) + \frac{n(\log n)^{2\gd_{j-1}}}{2n}\notag\\
& =\log\big(\,(\eps \ell^{-1})^2 \,n^{-a}/4\big) + \frac{1}{2}(\log n)^{1+\frac{2^{j-1}-1}{2}(2\eta-1) }\notag \\
&\ge  h_{j-1}^2/4n,
\end{align}
for every $j\ge 2$, since $\eta>1/2$. Moreover, we have 
\begin{align}\label{pbd21}
 |\vB_1|\bar{F}(k_n)&\le \text{const.} \,|\vB_1|\bar{F}(m(n^{3/2}(\log n)^{\eta}))\notag\\
& \le \text{const.}\, n^{-1/2}h_1\cdot (\log n)^{-\eta}=  \text{const.}\, (\log n)^{\frac{1}{4}-\frac{\eta}{2}}=o(1),
\end{align}
by the choice $\eta\in(1/2,\ga)$.
 Using \eqref{slowcomp} in the first inequality below and the definition of $m(\cdot)$ in the second equality, we also have
\begin{align}\label{pbd2}
	\gb_n |\vB_j|\E[\go\ind_{\{\go>k_n\}}] &\le \,\text{const.} \gb_n\,(nh_j)\,k_n^{1-\ga}L_1(k_n)
	\le \text{const.}\, nh_j \,\gb_nk_n \,\bar{F}(k_n)\notag\\
	& \le \text{const.}\, nh_j \,\gb_nk_n\, \bar{F}(m(n^{3/2} (\log n)^\eta))\notag\\
	&= \text{const.}  \, n h_j \,\,\frac{m(n^{3/2}(\log n)^\eta)}{m(n^{3/2})} \frac{1}{n^{3/2} (\log n)^\eta}\notag\\
	&=  \text{const.}\,\frac{h_{j-1}^2}{n}\, (\log n)^{\gd_j-2\gd_{j-1}-\eta} \,\,\,\frac{m(n^{3/2}(\log n)^\eta)}{m(n^{3/2})} ,
\end{align}
where in the last step we used the definition of $h_{j-1},h_j$. We will establish at the end of this proof that for any $\vartheta>0$, which we will choose to be small,
we have for large enough $n$ that 
\begin{equation}\label{kara}
\frac{m(n^{3/2}(\log n)^\eta)}{m(n^{3/2})} < (\log n)^{\frac{\eta}{(1-\vartheta)\ga}},
\end{equation}
and inserting this into \eqref{pbd2} we obtain
\begin{align}\label{pbd23}
	\gb_n |\vB_j|\E[\go\ind_{\{\go>k_n\}}]  \le \text{const.}\,\frac{h_{j-1}^2}{n}\, (\log n)^{\gd_j-2\gd_{j-1}-\eta(1-\frac{1}{(1-\vartheta)\ga})}  = o(1) \frac{h_{j-1}^2}{n},
\end{align}
where the last equality holds if we choose $\ga\in(1/2,2], \,1/2<\eta<\ga$ and $\vartheta$ appropriately small, since
\begin{align*}
\gd_j-2\gd_{j-1}-\eta\Big(1-\frac{1}{(1-\vartheta)\ga}\Big)&=
-\frac{3}{4}+\eta\frac{2-\ga(1-\vartheta)}{2\ga(1-\vartheta)},
\end{align*}
which for $\ga>1/2$ is bounded by
\begin{align*}
-\frac{3}{4}+\frac{2-\ga(1-\vartheta)}{2(1-\vartheta)}
=\frac{-(2\ga+3)(1-\vartheta)+4 }{4(1-\vartheta)} <0,
\end{align*}
for $\vartheta$ small enough. Inserting \eqref{pbd1denom2}, \eqref{pbd21} and \eqref{pbd23} into \eqref{eq:pbd1} we obtain
\begin{align*}
\pr(\log Z_{n,\gb_n}^\go  - \log  Z_{n,\gb_n}^{\tgo}> n^{-a}) &\le 
\text{const.}\, (\log n)^{\frac{1}{4}-\frac{\eta}{2}} +\pr\left(Z_{n,\gb_n}^{\tgo}< \eps \E[Z_{n,\gb_n}^{\tgo}]\right)\\
& \,\,+ \eps + 
\text{const.}\,\ell\,\, (\log n)^{\gd_j-2\gd_{j-1}-\eta(1-\frac{1}{(1-\vartheta)\ga})}.
\end{align*}
The choice of $\ell=\lceil\log_2(1+ \frac{2\log n}{(2\eta-1)\log\log n}) \rceil$, as well as of $(\gd_j)_{j\geq 1}$, $\eta\in(1/2,\ga)$ and (small) $\vartheta$, implies that
\begin{align*}
\limsup_{n\to\infty}\pr(\log Z_{n,\gb_n}^\go  - \log  Z_{n,\gb_n}^{\tgo}> n^{-a}) \le \eps + \limsup_{n\to\infty} \pr\left(Z_{n,\gb_n}^{\tgo}< \eps \E[Z_{n,\gb_n}^{\tgo}]\right),
\end{align*}
from which the result follows by taking $\eps\downarrow 0$. 
\vskip 2mm
It only remains to check the validity of \eqref{kara}. This will be done with the help of Karamata's theorem,
which states that any slowly varying function $L(\cdot)$ has the form $c(n)\exp\big( \int_{1}^n\eps(s)/s\, \dd s \big)$, where $c(\cdot)$ is an asymptotically constant function
and $\eps(\cdot)$ is asymptotically zero. Hence, by the definition of $m(\cdot)$ we have that
$$
\left(\frac{m\big(n^{3/2}(\log n)^\eta\big)}{m\big(n^{3/2}\big)}\right)^{-\ga} \frac{L\big( m\big(n^{3/2}(\log n)^\eta\big) \big) }{  L\big(m\big(n^{3/2}\big)\big) } =\frac{1}{(\log n)^\eta},
$$
which implies that
\begin{align*}
\frac{m\big(n^{3/2}(\log n)^\eta\big)}{m\big(n^{3/2}\big)}
&= (\log n)^\frac{\eta}{\ga} \,\left(\frac{L\big( m\big(n^{3/2}(\log n)^\eta\big) \big) }{  L\big(m\big(n^{3/2}\big)\big) } \right)^{\frac{1}{\ga}}\\
&\le \text{const.} \,(\log n)^\frac{\eta}{\ga}\, \exp\Big( \ga^{-1}\,\int_{m\big(n^{3/2}\big)}^{m\big(n^{3/2}(\log n)^\eta \big)} \eps(s)/s\, \dd s  \Big)\\
&< \text{const.} \,(\log n)^\frac{\eta}{\ga}\, \left(\frac{m\big(n^{3/2}(\log n)^\eta\big)}{m\big(n^{3/2}\big)}\right)^\vartheta,
\end{align*}
where in the last step we bound $\eps(s)$ by $\ga\vartheta$ with $\vartheta$ arbitrarily small, for all $s> m\big(n^{3/2}\big)$ and all $n$ large enough.
\end{proof}

\section{Proof of Theorem \ref{thm:14}}\label{sec:conv}
	 We need to establish that the limiting distribution of the (centered) partition function with truncated disorder converges to the desired quantity. We do this via a multilinear expansion of the
	 partition function and establish that these series converge. In order to check this
	  we may apply (a version of) Theorem~4.5 in~\cite{AKQ12} or Theorem 3.8 of \cite{CSZ15}:	
\begin{thm}\label{AKQ-CSZ}
 Let $(p(i,x))_{i\in\bN,x\in\bZ}$ be the transition kernel of a one-dimensional simple random walk, i.e. $p(i,x)=\vP(s_i=x)$, for $i\in \bN, x\in \bZ$. Let, also, $(\zeta_{i,x}^{(n)})_{i\in \bN,x\in \bZ}$ be a family of independent 
random variables, such that
\begin{itemize}
\item[$\bullet$] $\bE[\zeta^{(n)}]=0$,
\item[$\bullet$] $\var\big(\zeta^{(n)}\big)=1+o(1),\quad \text{as}\,\,\, n\to\infty$,
\item[$\bullet$] The family $((\zeta^{(n)})^2)_{n\ge1 }$ is uniformly integrable.
\end{itemize}
Then we have the convergence in distribution and in $L^2(\bP)$ of the multilinear series
\begin{align*}
1+\sum_{k=1}^n  
 \gb_n^k \sum_{\substack{1\leq i_1<\cdots< i_k\leq n\\ x_1,...,x_k\in \bZ}} 
 \prod_{j=1}^k p(i_j-i_{j-1},x_j-x_{j-1}) \,\, \zeta^{(n)}_{i_j,x_j} \longrightarrow \cZ_{\sqrt{2}\gb},
\end{align*}
whenever $\gb_nn^{1/4}\to \gb$, with $\cZ_{\sqrt{2}\gb}$ the Wiener chaos expansion \eqref{Wiener}.
\end{thm}	
\begin{prop}\label{trun-conv}
Assume that the weights satisfy  $\E[\go]=0,\E[\go^{2}]=1$ and $\pr(\go>x)=x^{-\ga}L(x)$ for some $\ga\ge6$ and some slowly varying function $L(x)$. Let $\gb_{n}$ be a sequence of real numbers with $\gb_{n}n^{1/4}\to \gb>0$ as $n\to\infty$ and $\tgo=\go\ind_{\go\le k_n}$ with $k_n=\gb_n^{-1}$, if $\ga>6$ and $k_n=\gb_n^{-1}m(n^{3/2}(\log n)^\eta)/m(n^{3/2})$ with $1/2<\eta<\ga$, if $\ga=6$. Then 
		\[
			\log Z_{n,\gb_n}^{\tgo} - n\log\E\Big(e^{-\gb_n\go_-} + \sum_{i=1}^{4}\frac{\gb_n^i}{i!}\go_+^i \Big)  \xrightarrow[n\to\infty]{(d)} \log \cZ_{\sqrt{2}\gb}.
		\]
\end{prop}
\begin{proof}
Let the truncated log-moment generating function be
$$
\gl_n(x):=\log \E[e^{x\tilde\go}].
$$
We will rewrite $Z_{n,\gb_n}^\tgo \,e^{-n\gl_n(\gb_n)}$ in the form of a multilinear polynomial
\begin{align}\label{multi-exp}
 e^{-n\gl_n(\gb_n)}\, Z_{n,\gb_n}^\tgo &= \Big\la \exp\Big(\sum_{i=1}^n ( \gb_n\go_{i,s_i} -\gl_n(\gb_n) \,)\Big)\Big\ra \notag\\
 &= 1+\sum_{k=1}^n\,\sum_{\substack{1\leq i_1<\cdots<i_k  \leq n \\ x_1,...,x_k \in \bZ}} \beta_n^k \prod_{j=1}^k p(i_j-i_{j-1},x_j-x_{j-1}) \,\, \zeta^{(n)}_{i_j,x_j},
\end{align}
where 
\begin{equation}\label{multi-var}
	\zeta_v^{(n)}:=\gb_n^{-1} (e^{\beta_n\tilde\go_v-\lambda_n(\gb_n)}-1) \text{ for } v\in\dZ^2
\end{equation}
It is immediate that $\E[\zeta_v^{(n)}]=0$ and using  Lemma~\ref{lem:comp} that 
\[
	\E\big[(\zeta_v^{(n)})^2\big] = \gb_n^{-2} e^{- 2\gl_n(\gb_n)}(e^{\gl_n(2\gb_n)}-e^{2\gl_n(\gb_n)})=1+O(\gb_{n}).
\]
(notice that the condition $e^{\gb_nk_n}\gb_n^{\ga-\theta}\to 0$ is satisfied, since $\gb_nk_n\le m(n^{3/2}(\log n)^\eta)/m(n^{3/2})\le (\log n)^{\eta/(1-\vartheta)\ga}<<\log n$, for $\eta<\ga$ and $\theta$ small).
Moreover, $\zeta^{(n)}$ have uniformly integrable second moments as the following computation, for $2<p<p'<\ga$, shows: Denote $\Vert\go\Vert_p:=(\E|\go|^p)^{1/p}$,
\begin{align*}
\Vert\zeta^{(n)}\Vert_p
& = \beta_n^{-1} \Vert e^{\gb_n\tgo-\gl_n(\beta_n)}-1\Vert_p \\
& \le  \beta_n^{-1} e^{-\gl_n(\beta_n)} \Vert e^{\beta_n\tgo}-1\Vert_p
                         + \beta_n^{-1} |1-e^{-\gl_n(\beta_n)}| \\
& \le e^{-\gl_n(\beta_n)} \Vert\tgo e^{\gb_n\tgo_+}\Vert_p
                         + \beta_n^{-1} |1-e^{-\gl_n(\beta_n)}| \\ 
&   \le  e^{-\gl_n(\beta_n)}\cdot \Vert \go\Vert_{p'}\cdot \Vert e^{\gb_n\tgo_+}\Vert_{pp'/(p'-p)}
+ \beta_n^{-1} |1-e^{-\gl_n(\beta_n)}|,
\end{align*}
where in the third inequality we used the fact that $|1-e^{x}|\le |x|\max\{1,e^x\}$, for all $x\in\dR$.  Lemma~\ref{lem:comp} shows that the last term in right hand side is uniformly bounded. Since $p'<\ga$, we
 also have that  $ \Vert \go\Vert_{p'} <\infty$ and it remains to check the boundedness of $ \Vert e^{\gb_n\tgo_+}\Vert_{pp'/(p'-p)}$. This follows immediately in the case $\ga>6$ since the truncation level
 equals $k_n=\gb_n^{-1}$ and so $\gb_n\tgo_+\le 1$. In the case $\ga=6$ the truncation is $k_n=\gb_n^{-1}m(n^{3/2}(\log n)^\eta)/m(n^{3/2})$. Denoting, for conciseness, $q=pp'/(p'-p)$ we have
 \begin{align*}
 \Vert e^{\gb_n\tgo}\Vert_q^q &=\E[e^{q\gb_n \tgo_+};\go\le\gb_n^{-1}] +\E[e^{q\gb_n \tgo_+};\go\ge\gb_n^{-1}] 
 \le e^q+ e^{q\gb_nk_n} \bar{F}(\gb_n^{-1})\\
 &\le e^q+\exp\Big( qm(n^{3/2}(\log n)^\eta)/m(n^{3/2}) \Big)\,\, \gb_n^{\ga} L(\gb_n^{-1}) 
 \end{align*}
Using relation \eqref{kara} and the assumption that $\gb_nn^{1/4}\to\gb>0$, we estimate the above, for all large enough $n$, by
\[
e^q+\exp\Big( q(\log n)^{\frac{\eta}{(1-\vartheta)a}} -\text{const.} \log n\Big),
\]
which is uniformly bounded since $\eta<\ga$ and $\vartheta$ can be chosen arbitrarily small.
Thus, the assumptions of Theorem \ref{AKQ-CSZ} are satisfied and we have that
\[
	 e^{-n\lambda_n(\gb_n)} \,Z_{n,\gb_n}^{\tgo}  \xrightarrow[n\to\infty]{(d)} \cZ_{\sqrt{2}\gb}\
\]
when $\gb_{n}n^{1/4}\to\gb>0$. Finally, by Lemma \ref{lem:comp} we check that, for $n\to\infty$ we have\\
\hspace*{35mm}$\displaystyle
	\Big|n\lambda_n(\gb_n) - n\log\E\Big(e^{-\gb_n\go_-} + \sum_{i=1}^{4}\frac{\gb_n^i}{i!}\go_+^i \Big) \Big| = n\, o(\gb_n^{4}) = o(1).
$
\end{proof}
We can now conclude the proof of Theorem~\ref{thm:14}.

\begin{proof}[Proof of Theorem~\ref{thm:14}]
Proposition \ref{prop:bound} states that 
$$
		n^{a}\big(\log Z_{n,\gb_n}^\go  - \log  Z_{n,\gb_n}^{\tgo}\big)\xrightarrow[n\to\infty]{\bP} 0,
$$
for any $a\ge 0$, under the assumption that
\begin{equation}\label{posi-est2}
\lim_{\eps\,\downarrow \,0}\limsup_{n\to\infty}\pr\left(Z_{n,\gb_n}^{\tgo}< \eps \E[Z_{n,\gb_n}^{\tgo}]\right) = 0.
\end{equation}
However, $Z_{n,\gb_n}^{\tgo} / \E[Z_{n,\gb_n}^{\tgo}]= e^{-\gl_n(\gb_n) }Z_{n,\gb_n}^{\tgo}$ converges in distribution to $\cZ_{\sqrt{2}\gb}$, when $\gb_nn^{1/4}\to\gb>0$, which is a.s. strictly positive \cite{M91} and so the assumption is readily checked.
\vskip 2mm
When $\gb_nn^{1/4}\to 0$ we can see from \eqref{multi-exp} that 
$$
e^{-\gl_n(\gb_n) }Z_{n,\gb_n}^{\tgo} = 1 +(\gb_nn^{1/4})\,\,\frac{1}{n^{3/4}} \sum_{1\le i\le n, x\in \bZ} (\sqrt{n}p(i,x)) \,\, \zeta^{(n)}_{i,x} +o(\gb_nn^{1/4}),
$$
where the $o(\gb_nn^{1/4})$ is in an $L^2(\bP)$ sense. Thus, clearly, assumption \eqref{posi-est2} is satisfied. Moreover,
$$
\frac{1}{n^{3/4}} \sum_{1\le i\le n, x\in \bZ} \zeta^{(n)}_{i,x}\,(\sqrt{n}p(i,x)) \xrightarrow[n\to\infty]{(d)} \sqrt{2}\int_{(0,1)\times \bR} \rho(t,x) W(\dd t\dd x),
$$
which is a mean zero Gaussian with variance $2\pi^{-1/2}$.
\end{proof}

\section{Proof of Theorem \ref{thm:gauss}}		
 \begin{prop}\label{truncated-gauss}
	Assume that the weights satisfy  $\E[\go]=0,\E[\go^{2}]=1$ and $\pr(\go>x)=x^{-\ga}L(x)$ for some $\ga\in (2,6]$ and some slowly varying function $L(x)$. Let $\gb_{n}$ be a sequence of real numbers such that $\gb_{n} m(n^{3/2}) $ stays bounded, as $n\to\infty$, but $\gb_nn^{1/4}$ converges to zero. Let $\tgo=\go \ind_{\go\leq k_n}$, with 
	$k_n=\gb_n^{-1}\,m(n^{3/2}(\log n)^{\eta})/ m(n^{3/2}) $ and $\eta\in(1/2,\ga)$. Then 
	\[
		\frac{1}{\gb_n n^{1/4}}\Biggl(\log Z_{n,\gb_n}^{\tgo} - n\log\E\Big(e^{-\gb_n\go_-} + \sum_{i=1}^2 \frac{\gb_n^i}{i!}\go_+^i + \frac{\gb_n^3}{3!}\go_+^3\ind_{\ga>3}\Big) \Biggr) \xrightarrow[n\to\infty]{(d)} \cN(0,2\pi^{-1/2}) 
			\]
\end{prop}
 \begin{proof}
 We denote, again, the truncated log-moment generating function
$$
\gl_n(x):=\log \E[e^{x\tilde\go}],
$$
with $\tgo:=\go\ind_{\go\le k_n}$ and $k_n=\gb_n^{-1}\,m(n^{3/2}(\log n)^{\eta})/ m(n^{3/2}) $.
Write $Z_{n,\gb_n}^\tgo \,e^{-n\gl_n(\gb_n)}$ in the form of a multilinear polynomial as in \eqref{multi-exp}, \eqref{multi-var}. 
Denoting by $\gz^{(n)}_{i,x}:=\gb_n^{-1}(e^{\gb_n\tgo_{i,x}-\gl_n(\gb_n)}-1)$, we have
\begin{align}\label{411}
&\gb_n^{-1}n^{-\frac{1}{4}}\Big(e^{-n\gl_n(\gb_n)}\, Z_{n,\gb_n}^\tgo -1\Big) -n^{-\frac{1}{4}}\sum_{1\le i\le n,\, x\in \bZ} p(i,x) \zeta^{(n)}_{i,x}\notag\\
&\qquad\qquad= \sum_{k=2}^n  
 c_{n,k}\,\,\, n^{-\frac{k}{4}}  \sum_{\substack{1\leq i_1<\cdots< i_k\leq n\\ x_1,...,x_k\in \bZ}} 
 \prod_{j=1}^k p(i_j-i_{j-1},x_j-x_{j-1}) \,\, \zeta^{(n)}_{i_j,x_j},
\end{align}
where 
$c_{n,k}:=(\beta_n n^{1/4})^{k-1} \longrightarrow 0$
since $\gb_nn^{1/4}$ converges to zero and $k\ge 2$. 
The estimates in the proof of Proposition \ref{trun-conv} show that
$$\sum_{k=2}^n  
n^{-k/4}  \sum_{\substack{1\leq i_1<\cdots< i_k\leq n\\ x_1,...,x_k\in \bZ}} 
 \prod_{j=1}^k \,p(i_j-i_{j-1},x_j-x_{j-1}) \,\, \zeta^{(n)}_{i_j,x_j},
$$
is bounded in $L^2(\bP)$. Therefore the right hand side of \eqref{411} converges to zero in $L^2(\bP)$. Moreover,
$$n^{-1/4}\sum_{1\le i\le n,\, x\in \bZ} p(i,x) \zeta^{(n)}_{i,x} \xrightarrow{(d)} \cN(0,2\pi^{-1/2}).
$$ 
Noticing that $(\gb_nn^{1/4})^{-1}\log (e^{-n\gl_n(\gb_n)}\, Z_{n,\gb_n}^\tgo\,)\approx (\gb_nn^{1/4})^{-1}(e^{-n\gl_n(\gb_n)}\, Z_{n,\gb_n}^\tgo -1\,) $, the result follows once we 
 check the asymptotic behavior of the centering $n\gl_n(\gb_n)$. To this end, we invoke, again, Lemma \ref{lem:comp} and get, for any $\theta<\ga$,
\begin{align*}
(\gb_n n^{1/4})^{-1}\,\,\Big| n\gl_n(\gb_n) -n\log\E\Big(e^{-\gb_n\go_-} + \sum_{i=1}^{\lfloor \theta \rfloor} \frac{\gb_n^i}{i!}\go_+^i \Big)\,\Big|
 &= (\gb_n n^{1/4})^{-1} \,n\,o(\gb_n^\theta) \\
&= o(n^{3/4}\gb_n^{\theta-1})
=o(n^{\frac{6-6\theta+3\ga}{4\ga} })=o(1),
\end{align*}
since $\ga>2$ and $\theta$ can be chosen to be arbitrarily close to $\ga$ (Notice that in the previous display we ignored slowly varying corrections, since these are immaterial). Finally, if $\ga>4$ and hence we choose $\theta>4$, then
\begin{align*}
&(\gb_n n^{1/4})^{-1} \,n\,\Big| \log\E\Big(e^{-\gb_n\go_-} + \sum_{i=1}^{\lfloor \theta \rfloor} \frac{\gb_n^i}{i!}\go_+^i \Big) - \log\E\Big(e^{-\gb_n\go_-} + \sum_{i=1}^2 \frac{\gb_n^i}{i!}\go_+^i + \frac{\gb_n^3}{3!}\go_+^3\ind_{\ga>3}\Big)\,\Big| \\
&= (\gb_n n^{1/4})^{-1} \,n \,O(\gb_n^4) = O\big((\gb_nn^{1/4})^3\big) =o(1),
\end{align*}
whenever $\gb_nn^{1/4}$ converges to zero. If $\ga\leq 4$, then $\theta$ is chosen strictly less that $4$ and the above difference is trivially equal to zero. 
 \end{proof}
 
 \begin{proof}[Proof of Theorem \ref{thm:gauss}]
 The proof follows immediately from Propositions \ref{truncated-gauss} and \ref{prop:bound} once we check that assumption \eqref{posi-est} in Proposition \ref{prop:bound} is satisfied.
 But this is clear from relation \eqref{411}, which implies that $e^{-n\gl_n(\gb_n)}\, Z_{n,\gb_n}^\tgo=1+O(\gb_n n^{1/4})=1+o(1)$.
 \end{proof}
\section{Proof of theorem \ref{thm:heavy}}\label{sec:le2}
Note that we already have, by Proposition \ref{prop:bound}, that
\[
n^{1/2}(\log Z_{n,\gb_{n}}^{\go} - \log Z_{n,\gb_{n}}^{\tgo} ) \xrightarrow[n\to\infty]{\bP}0,
\]
where $\tgo=\go\ind_{\go\le k_{n}}$ with $k_{n}=\gb_{n}^{-1}m(n^{3/2}(\log n)^{\eta})/m(n^{3/2})$ for some $\eta\in(1/2,\ga)$. So, as in the previous cases, it suffices to determine the limit of
the partition function with truncated disorder, $Z_{n,\gb_{n}}^{\tgo}$. This will be the main effort in the proof. 
Below, we will restrict attention to the case $\gb_n m(n^{3/2})\to \gb>0$. The case $\gb=0$ follows by plain inspection of the bounds in the proof. This is also reflected in the fact that
$\cW_{\gb}^{(\ga)}  \xrightarrow[n\to\infty]{(d)}\cW_{0}^{(\ga)} $, since for $\beta$ tending to zero
\begin{align*}
\frac{1}{\gb} \int_{\sS} (e^{\gb w}-1-\gb w)\rho(t,x)\cP(\dd w \dd t \dd x) \sim \gb \int_{\sS} w^2 \rho(t,x)\cP(\dd w \dd t \dd x),
\end{align*}
which tends to zero, since the last integral is $\cP$-a.s.~finite, thanks to Lemma \ref{PoissonInt},

\begin{proof}[Proof of Theorem \ref{thm:heavy}]
As explained, we assume $\gb_n m(n^{3/2})\to \gb>0$.
Performing the usual multilinear expansion we have
\begin{align*}
e^{-n\gl_n(\gb_n)} Z_{n,\gb_{n}}^{\tgo} = 1 + \sum_{i,x}(e^{\gb_n\tgo_{i,x}-\gl_n(\gb_n)}-1) p_n(i,x) + R_n,
\end{align*}
where the remainder $R_n$ equals
\begin{align*}
R_n:= \sum_{k=2}^\infty\, \sum_{\substack{1\leq i_1<\cdots< i_k\leq n\\ x_1,...,x_k\in \bZ}} \, \prod_{j=1}^k\, (e^{\gb_n\tgo_{i_j,x_j}-\gl_n(\gb_n)}-1) \,p_n(i_j-i_{j-1},x_j-x_{j-1}). 
\end{align*}
We first show that $n\bE[R_n^2]=o(1)$, as $n$ tends to infinity. Using the fact that $e^{\gb_n\tgo_{v}-\gl_n(\gb_n)}-1$ are mean zero we have
\begin{align}\label{L2Rn}
n\bE[R_n^2] =n \sum_{k=2}^\infty\, \sum_{\substack{1\leq i_1<\cdots< i_k\leq n\\ x_1,...,x_k\in \bZ}} \, (e^{\gl_n(2\gb_n) -2\gl_n(\gb_n)}-1)^k \,\prod_{j=1}^k p_n^2(i_j-i_{j-1},x_j-x_{j-1}).
\end{align}
From the second part of Lemma \ref{lem:comp} we have that
\begin{align*}
e^{\gl_n(2\gb_n) -2\gl_n(\gb_n)}-1& \le \text{const.} \,e^{\gb_nk_n} \bar{F}(k_n) \leq \text{const.} \,e^{ m(n^{3/2}(\log n)^\eta)/ m(n^{3/2}) } \,\bar{F}(m(n^{3/2}(\log n)^\eta))\\
&\le \text{const.} \,e^{(\log n)^{\eta/(1-\vartheta)\ga}} n^{-3/2} (\log n)^{-\eta} \le \text{const.} \,e^{(\log n)^{\eta/(1-\vartheta)\ga}} n^{-3/2},
\end{align*}
for all large $n$ and $\vartheta$ arbitrarily small, as in \eqref{kara}.
Inserting this into \eqref{L2Rn} we obtain
\begin{align*}
n\bE[R_n^2] &\le n \sum_{k=2}^\infty  (\text{const.}\,e^{(\log n)^{\eta/(1-\vartheta)\ga}} n^{-3/2})^k \,\sum_{\substack{1\leq i_1<\cdots< i_k\leq n\\ x_1,...,x_k\in \bZ}}   \,\prod_{i=1}^kp_n(i_j-i_{j-1},x_j,x_{j-1})^2\\
&=n \sum_{k=2}^\infty  (\text{const.}\,e^{(\log n)^{\eta/(1-\vartheta)\ga}} n^{-3/2})^k n^{-k} \,\sum_{\substack{1\leq i_1<\cdots< i_k\leq n\\ x_1,...,x_k\in \bZ}}\,   \,\prod_{i=1}^k (\sqrt{n}p_n(i_j-i_{j-1},x_j,x_{j-1}))^2.
\end{align*}
Since the summation runs over $k\ge 2, \eta<\ga$ and $n^{-3k/2} \sum_{\substack{1\leq i_1<\cdots< i_k\leq n\\ x_1,...,x_k\in \bZ}}   \,\prod_{i=1}^k (\sqrt{n}p_n(i_j-i_{j-1},x_j-x_{j-1}))^2  $ converges to a finite Riemann integral, it is easily seen that 
$n\bE[R_n^2]$ converges to zero, as $n$ tends to infinity. 

To proceed further, we write
\begin{align*}
 Z_{n,\gb_{n}}^{\tgo} &= e^{n\gl_n(\gb_n)}\bigg(1 + \sum_{1\leq i\leq n , x\in\mathbb{Z}}(e^{\gb_n\tgo_{i,x}-\gl_n(\gb_n)}-1) p_n(i,x) + R_n\bigg)\\
 &= e^{(n-1)\gl_n(\gb_n)}\bigg( e^{\gl_n(\gb_n)} + \sum_{1\leq i\leq n ,x\in \mathbb{Z}}(e^{\gb_n\tgo_{i,x}}-e^{\gl_n(\gb_n)}) p_n(i,x) + e^{\gl_n(\gb_n)}R_n\bigg)\\
 &= e^{(n-1)\gl_n(\gb_n)}\bigg( 1+ \sum_{1\leq i\leq n , x\in\mathbb{Z}}(e^{\gb_n\tgo_{i,x}}-1) p_n(i,x) - (n-1)\big(e^{\gl_n(\gb_n)}-1\big) + e^{\gl_n(\gb_n)}R_n\bigg)\\
\end{align*}
As it will be checked, the terms inside the parenthesis, besides $1$, tend to zero, in probability,  as $n$ tends to infinity and in fact the third and fourth terms are $o(n^{-1/2})$, while the second term will be of order $O(n^{-1/2})$. Hence,
\begin{align*}
\sqrt{n}\log &Z_{n,\gb_{n}}^{\tgo} \\
&=\sqrt{n} (n-1) \gl_n(\gb_n) + \sqrt{n} \log \Big( 1+ \sum_v(e^{\gb_n\tgo_v}-1) p_n(v) - (n-1)\big(e^{\gl_n(\gb_n)}-1\big) + e^{\gl_n(\gb_n)}R_n\Big)\\
  &=\sqrt{n} \sum_v(e^{\gb_n\tgo_v}-1) p_n(v) - \sqrt{n} (n-1) \Big( e^{\gl_n(\gb_n)}-1-\gl_n(\gb_n)  \Big) + \sqrt{n} e^{\gl_n(\gb_n)}R_n +o(1)\\
  & =\sqrt{n} \sum_v(e^{\gb_n\tgo_v}-1) p_n(v) - \sqrt{n} (n-1) \Big( e^{\gl_n(\gb_n)}-1-\gl_n(\gb_n)  \Big) +o(1).
 \end{align*}
 Using the second part of Lemma \ref{lem:comp}, we have the estimate
 \begin{align*}
 \sqrt{n} (n-1) \Big| e^{\gl_n(\gb_n)}-1-\gl_n(\gb_n)  \Big|  
 &\le \text{const.} \,n^{3/2} (\gl_n(\gb_n))^2 \\
 &=\text{const.}\,n^{3/2} \big( \log(1 + \bE[e^{\gb_n\go\ind_{\go\le k_n}}]-1) \,\big)^2\\
 &\leq\text{const.}\,n^{3/2} \big(\bE[e^{\gb_n\go\ind_{\go\le k_n}}]-1) \,\big)^2\\
 &\le \text{const.}\,n^{3/2} e^{2\gb_n k_n} (\bar{F}(\gb_n^{-1}))^2\\
 &\le \text{const.}\,n^{3/2} e^{2\gb_n k_n} \big(\bar{F}(m(n^{3/2}))\big)^2\\
 &= \text{const.}\,n^{-3/2} \,e^{2\gb_n k_n} ,
 \end{align*}
 which converges to zero as $n$ tends to infinity, {since $\gb_n k_n=m(n^{3/2}(\log n)^\eta)/m(n^{3/2})$ and $\eta\in(1/2,\ga)$}. Therefore, to identify the distributional limit of $\sqrt{n}\log Z_{n,\gb_{n}}^{\tgo} $ it remains to do so for $\sqrt{n} \sum_{v}(e^{\gb_n\tgo_v}-1) p_n(v) $.
 For this, we first notice that the field $\big\{(n^{-1}i,n^{-1/2}x,m(n^{3/2})^{-1}\go_{(i,x)}) \colon i+x\,\,\text{is even}\,\,,|x|\le K\sqrt{n}\big\}$ 
 converges to a Poisson field $\cP$ on $\bR\times(0,1)\times (-K,K)$ with intensity measure $\frac{1}{2}\ga |w|^{-(1+\ga)} (\ind_{w> 0} +c_-\ind_{w<0})\dd w\dd t\dd x$. 
To see this, denote for any set $A\subset (0,1)\times \bR$ and (without loss of generality) any positive $r$
\begin{equation}\label{NA_Poi}
N_{A}:=\sharp\big\{(i,x)\in\bN\times \bZ  \colon\,i+x\,\,\text{is even},\,\, (n^{-1}i,n^{-1/2}x)\in A \,\,\,\text{and}\,\, \,m(n^{3/2})^{-1}\go_{(i,x)}>r  \big\}.
\end{equation}
For every single $(i,x)\in\bN\times \bZ$ it holds $\bP(m(n^{3/2})^{-1}\go_{(i,x)}>r)\to0,$ as $n\to\infty$, while
\begin{align*}
\bE[ N_A ] 
&= \frac{1}{2}n^{3/2}(1+o(1))\,|A| \bar{F}(rm(n^{3/2}))
=  \frac{1}{2} n^{3/2}(1+o(1))\,\,|A| \,\,r^{-\ga} m(n^{3/2})^{-\ga} L(rm(n^{3/2}))\\
&=\,  \frac{1}{2} n^{3/2}(1+o(1))\,|A| \,r^{-\ga}  \bar{F}(m(n^{3/2}))=(1+o(1))\,|A| \,r^{-\ga}  
=\, \frac{1}{2}(1+o(1)) \int_A\, \int_{r}^\infty \frac{\ga\dd w}{w^{1+\ga}}\dd t \dd x,
\end{align*}
where the factor $1/2$ comes from the parity condition `$i+x$ is even' and we also used the defining property of slowly varying functions, i.e. $L(ry)/L(y)\to 1$, for every $r\in \bR$ and $y\to\infty$. Hence, it follows that $N_A$ is a Poisson random variable. To complete the check that the field $(n^{-1}i,n^{-1/2}x,m(n^{3/2})^{-1}\go_{(i,x)})$ converges to a Poisson field, it suffices to check that for any two disjoint sets $A_1\times(a_1,b_1)$ and $A_2\times(a_2,b_2)$, where $A_1,A_2\subset (0,1)\times \mathbb{R}$ and w.l.o.g. $(a_1,b_1), (a_2,b_2)\subset \mathbb{R}_+$, the random variables
\[
N_{A_r,a_r,b_r}:=\sum_{(i,x)\colon i+x \,\text{is even}} \ind_{\big\{(n^{-1} i,n^{-1/2} x, m(n^{3/2})^{-1} \omega_{(i,x)})  \,\in\, A_r\times (a_r,b_r)\big\}},
\qquad r=1,2,
\] 
are asymptotically distributed as independent Poisson variables. This is clear in the case that $A_1, A_2$ are disjoint, by the independence of the random variables $\omega_{(i,x)}$ and the above computation on $N_A$. So, let us assume that
$A_1=A_2=A$ and the intervals $(a_1,b_1), (a_2,b_2)$ are disjoint and compute
\begin{align*}
&\mathbb{E}\big[e^{\lambda_1 N_{A,a_1,b_1}+\lambda_2 N_{A,a_2,b_2}} \big] =
\mathbb{E} \Big[\,e^{\lambda_1 \ind_{\{ \omega\in m(n^{3/2}) (a_1,b_1) \}}+
\lambda_2 \ind_{\{\omega\in m(n^{3/2}) (a_2,b_2) \}}} \Big]^{\frac{1}{2}n^{-3/2}|A| (1+o(1))}\\
&=\Big(1 +\big(e^{\lambda_1}-1\big)  \mathbb{P}\big(   \omega\in m(n^{3/2}) (a_1,b_1) \big) +
\big(e^{\lambda_2}-1\big)  \mathbb{P}\big(   \omega\in m(n^{3/2}) (a_1,b_2) \big)
\Big)^{\frac{1}{2}n^{-3/2}|A| (1+o(1))},
\end{align*}
and an easy computation shows that the latter converges to
\[
\exp\left( (e^{\lambda_1}-1) \frac{1}{2}\int_A \dd t \dd x\int_{a_1}^{b_1}\frac{\alpha \dd w}{w^{1+\alpha}} 
+  (e^{\lambda_2}-1) \frac{1}{2}\int_A \dd t\dd x\int_{a_2}^{b_2}\frac{\alpha \dd w}{w^{1+\alpha}} 
\right),
\]
from which the result follows.
\medskip 

We now proceed to identify the limit distribution of $\sqrt{n} \sum_{v}(e^{\gb_n\tgo_v}-1) p_n(v) $. We distinguish three cases
\vskip 2mm
{\bf Case 1. ($1/2<\ga<1$)} Denote by $x(v)$ the $x$ coordinate of $v=(i,x)\in \{0,1,...,n\}\times \bZ$ and write
\[
\sqrt{n} \sum_{v}(e^{\gb_n\tgo_v}-1) p_n(v) = \sqrt{n} \sum_{ |x(v)|<K\sqrt{n}}(e^{\gb_n\tgo_v}-1) p_n(v) +\sqrt{n} \sum_{|x(v)|>K\sqrt{n}}(e^{\gb_n\tgo_v}-1) p_n(v) 
\]
 Lemma \ref{Kinfty} below will show that the second term converges in probability to zero, as $K\nearrow\infty$, uniformly
in $n$. We, therefore, concentrate on the first term. To this, consider a partition $\sP_\gd$ of $\bR\times[0,1]\times \bR$ into disjoint rectangles of diameter $\gd>0$. For any $\pi\in \sP_\gd$ denote by $(w_\pi,t_\pi,x_\pi)$ its centre and write the first term as
\begin{align*}
\sqrt{n} \sum_{|x(v)|<K\sqrt{n}} &(e^{\gb_n\tgo_v}-1) p_n(v) \\
&= \sum_{\pi\in\sP_\gd}  \sum_{1\le i\le n,|x|\le K\sqrt{n}}(e^{\gb_n\tgo_{i,x}}-1)\,\, \sqrt{n}\,\,p_n\Big(n\frac{i}{n}, \sqrt{n}\frac{x}{\sqrt{n}}\Big)\,\,
 \ind_{\big(\frac{\go_{i,x}}{m(n^{3/2})},\frac{i}{n}, \frac{x}{\sqrt{n}} \big) \in\pi}\\
 &= \sum_{\pi\in\sP_\gd}  (1+o_\gd(1))\,(e^{\gb w_\pi}-1)\,\, 2\rho(t_\pi,x_\pi) 
  \sum_{\substack{1\le i\le n,|x|\le K\sqrt{n} \\ i+x\text{ is even}} } \,\,
 \ind_{\big(\frac{\go_{i,x}}{m(n^{3/2})},\frac{i}{n}, \frac{x}{\sqrt{n}} \big) \in\pi}
\end{align*}
where we used the fact that $\gb_n m(n^{3/2})\to\gb$ and the local limit theorem for the convergence $\sqrt{n}p_n(\cdot) \to 2\rho(\cdot)$.  
The notation $o_\gd(1)$ is used to denote errors that are negligible as $\gd\searrow 0$ due to the continuity of the functions involved in the expression. Moreover,
\[
\sum_{\substack{1\le i\le n,|x|\le K\sqrt{n} \\ i+x\text{ is even}}}\,\, \ind_{\big(\frac{\go_{i,x}}{m(n^{3/2})},\frac{i}{n}, \frac{x}{\sqrt{n}} \big) \in\pi}\longrightarrow \cP(\pi),
\]
as $n$ tends to infinity, where the Poisson measure $\cP(\pi)$ has intensity
 $\eta(\dd w \dd t \dd x) =\frac12\ga |w|^{-1-\ga}(\ind_{w>0}+c_-\ind_{w<0})\dd w \dd t \dd x$.
 
  Therefore, the above expression converges in the limit $n\to\infty$ followed by the limit $\gd\searrow 0$ to 
 \[
2 \int_{\bR\times(0,1)\times(-K,K)}(e^{\gb w}-1)\,\, \rho(t,x) \cP(\dd w\dd t\dd x).
 \]
 Taking the limit $K\nearrow\infty$ and thanks to Lemma \ref{PoissonInt} we obtain the limit
\begin{align*}
& 2\int_{\bR\times(0,1)\times \bR}(e^{\gb w}-1)\,\, \rho(t,x) \cP(\dd w\dd t\dd x) \\
&= 2\int_{\bR\times(0,1)\times \bR}(e^{\gb w}-1-\gb w)\,\, \rho(t,x) \cP(\dd w\dd t\dd x) + 2 \gb \int_{\bR\times(0,1)\times \bR}  w\,\, \rho(t,x) \cP(\dd w\dd t\dd x)
 \end{align*}
\vskip 2mm 
 {\bf Case 2. ($1<\ga < 2$)}
 Denoting by $\tgo_v^\eps:=\go_v \ind_{|\go_v|\le \eps m(n^{3/2})}=\tgo_v \ind_{|\go_v|\le \eps m(n^{3/2})}$, for $\eps$ small and $n$ large, we decompose
\begin{align}\label{decoC2}
\sqrt{n}\sum_{v}&(e^{\gb_{n}\tgo_v}-1)p_{n}(v) \notag\\
&= \sqrt{n}\sum_{|x(v)|\le K\sqrt{n}, |\go_v|\ge \eps m(n^{3/2})}(e^{\gb_{n}\tgo_v}-1)p_{n}(v) +\sqrt{n}\gb_n\E[\tgo_v^\eps]\sum_{|x(v)|\le K\sqrt{n}}p_{n}(v) \notag\\
&\qquad+ \sqrt{n}\sum_{|x(v)|\le K\sqrt{n}}(e^{\gb_{n}\tgo_v^\eps }-\E[e^{\gb_n\tgo_v^\eps}])p_{n}(v) \\
&\qquad+\sqrt{n}\sum_{|x(v)|> K\sqrt{n}}(e^{\gb_{n}\tgo_v}-1)p_{n}(v) + \sqrt{n}\E[e^{\gb_n\tgo_v^\eps}-1-\gb_n\tgo_v^\eps]\sum_{|x(v)|\le K\sqrt{n}}p_{n}(v) \notag
\end{align}

Moreover, a simple computation, {using the fact that $\omega$ has mean zero}, shows that
\begin{align*}
n^{3/2}\gb_n\E[\tgo_v^\eps] 
= -n^{3/2}\gb_n\E[\go_v \ind_{|\go_v|\ge \eps m(n^{3/2})}]
\xrightarrow[n\to\infty]{} 
& - \gb\ga \int_{|w|>\eps} |w|^{-\ga} (\ind_{w> 0} -c_-\ind_{w<0}) \dd w\\
=& - \gb\ga \int_{|w|>\eps} w\frac{1}{|w|^{1+\ga}} (\ind_{w> 0} +c_-\ind_{w<0}) \dd w .
\end{align*}
Notice that the integral is well defined since $\ga>1$.
In combination with the local limit theorem we see that
\begin{align*}
\sqrt{n}\gb_n\E[\tgo_v^\eps]\sum_{|x(v)|\le K\sqrt{n}}p_{n}(v) 
&\xrightarrow[n\to\infty]{} 
-\gb\ga \int_{|w|>\eps} w\frac{1}{|w|^{1+\ga}} (\ind_{w> 0} +c_-\ind_{w<0}) \dd w \int_{(0,1)\times(-K,K)} \rho(t,x) \dd t\dd x\\
&\qquad=-2\gb \int_{A_{\eps,K}}  w\rho(t,x) \eta(\dd w\dd t\dd x)
\end{align*}
where $A_{\eps,K}:=[-\eps,\eps]^{c}\times(0,1)\times(-K,K)$.
This fact together with the Poisson convergence (as in Case 1, cf. \eqref{NA_Poi}) implies that the first line of \eqref{decoC2} converges, as $n$ tends to infinity, to 
\begin{align*}
&2\int_{A_{\eps,K}}  (e^{\gb w}-1)\gr(t,x) \cP(\dd w \dd t \dd x ) - 2\gb\int_{A_{\eps,K}}  w\gr(t,x) \eta(\dd w \dd t \dd x ) \\
&\qquad=2\int_{A_{\eps,K}}  (e^{\gb w}-1-\gb w)\gr(t,x) \cP(\dd w \dd t \dd x ) + 2\gb\int_{A_{\eps,K}}  w\gr(t,x) \big(\cP(\dd w \dd t \dd x )- \eta(\dd w \dd t \dd x ) \big).
\end{align*}
Notice that in the last step we centered appropriately, in order to be able to take the limit $\eps\searrow 0$ in a legitimate, {$L^2(\cP)$, way}. This is done using the Poisson $L^2$ isometry (cf.~\cite{K02}, Thm. 10.15), since for any arbitrary $M>0$ we have
\begin{align*}
&\E^\cP \left( \int_{\{\eps<|w|<M\} \times(0,1)\times(-K,K) }  \,\,w\gr(t,x) \big(\cP(\dd w \dd t \dd x )- \eta(\dd w \dd t \dd x ) \big) \right)^2  \\
&\qquad=  \int_{\{\eps< |w|< M\}\times(0,1)\times(-K,K) }  w^2\gr(t,x)^2 \eta(\dd w \dd t \dd x ),
\end{align*}
which is uniformly bounded as $\eps$ tends to zero, thanks to the fact that $\ga<2$. Similarly, using an $L^1$ estimate, we have 
\begin{align*}
&\E^\cP \int_{\{\eps<|w|<M\}\times(0,1)\times(-K,K) } |e^{\gb w}-1-\gb w|\gr(t,x) \cP(\dd w \dd t \dd x ) \\
&\qquad\qquad \le\text{const.}  \int_{\{\eps<|w|<M\} \times(0,1)\times(-K,K) }  w^2\gr(t,x) \eta(\dd w \dd t \dd x ),
\end{align*}
which is, again, uniformly bounded as $\eps$ tends to zero, for $\alpha<2$.
Lemma \ref{Kinfty} below will allow us to take the limit
 $K\nearrow \infty$ and finally get convergence to
\begin{align*}
2\int_{\bR\times (0,1) \times \bR} (e^{\gb w}-1-\gb w)\gr(t,x) \cP(\dd w \dd t \dd x ) +2 \gb\int_{\bR\times (0,1)\times \bR} w\gr(t,x) \big(\cP(\dd w \dd t \dd x )- \eta(\dd w \dd t \dd x ) \big).
\end{align*}
To conclude we need to check that the rest of the terms in \eqref{decoC2} are negligible when $n$ tends to infinity and $\eps\searrow 0,K\nearrow\infty$. For the first term in the third line this follows from Lemma \ref{Kinfty}, below.
For the term in the second line this follows {by an $L^2(\bP)$ estimate, using the second part of Lemma \ref{lem:comp} with $k_n=\eps m(n^{3/2})=\eps \beta_n^{-1}$.}
 Regarding the last term, it suffices to show that $n^{3/2} \gb_n^2\,\E[(\tgo_v^\eps)^2]$ converges to zero, since $n^{-3/2} \sum_{|x(v)|\le K\sqrt{n}} \sqrt{n}p_{n}(v) $ converges to a finite Riemann integral
and $ \bE[|e^{\gb_n\tgo^\eps_v} -1 -\gb_n\tgo^\eps_v |] \leq \text{const.} \gb_n^2 \,\E[(\tgo_v^\eps)^2]$. To this end we have,
\begin{align*}
n^{3/2} \gb_n^2\,\E[(\tgo_v^\eps)^2] =n^{3/2}\gb_n^2\Big( \int_0^{\eps m(n^{3/2})} x^2 \dd F(x)  + \int^0_{-\eps m(n^{3/2})} x^2 \dd F(x)  \Big).
\end{align*}
Thanks to assumption \eqref{defC} it suffices to estimate the first term and we have 
\begin{align*}
n^{3/2}\gb_n^2 \int_0^{\eps m(n^{3/2})} x^2 \dd F(x) &= n^{3/2}\gb_n^2 \Big(\,-(\eps m(n^{3/2}))^2 \bar{F}(\eps m(n^{3/2})) + 2 \int_0^{\eps m(n^{3/2})} x^{1-\ga} L(x) \dd x\,\Big)\\
&\leq   2n^{3/2}\gb_n^2  \int_0^{\eps m(n^{3/2})} x^{1-\ga} L(x) \dd x\\
&\leq   \text{const.}\,n^{3/2}\gb_n^2 \,(\eps m(n^{3/2} ))^{2-\ga} L( \eps m(n^{3/2}   ))\\
&\leq   \text{const.}\,\eps^{2-\ga}\,n^{3/2} \,(\gb_n m(n^{3/2}))^2 \,\bar{F}(m(n^{3/2}))  \\
&= \text{const.}\,\eps^{2-\ga}  \,(\gb_n m(n^{3/2}))^2
\end{align*}
Since $\gb_nm(n^{3/2})\to \gb$ and $\ga<2$, the last is easily seen to converge to $0$, as $\eps\searrow 0$, uniformly for all large $n$.

\vskip 2mm
{\bf Case 3. ($\ga=1$)} Even though we work with the assumption that $\ga=1$, we will still use the general symbol $\ga$ in the computations below. We make a similar decomposition as in \eqref{decoC2}, with $\eps=1$, and denote $\hat\go_v=\tgo_v^1 = \go_v \ind_{|\go_v|\le  m(n^{3/2})}$.
\begin{align}\label{eq:Case3}
&\sqrt{n}\sum_{v}(e^{\gb_{n}\tgo_v}-1)p_{n}(v) - n^{3/2}\gb_n\E[\hat\go_v]\notag\\
&= \sqrt{n}\sum_{|x(v)|\le K\sqrt{n}, |\go_v|\ge  m(n^{3/2})}(e^{\gb_{n}\tgo_v}-1)p_{n}(v) +n^{3/2}\gb_n\E[\hat\go_v] \left(n^{-1}\sum_{|x(v)|\le K\sqrt{n}}p_{n}(v) -1\right)\notag\\
&\qquad+ \sqrt{n}\sum_{|x(v)|\le K\sqrt{n}}(e^{\gb_{n}\hat\go_v }-1-\gb_n\hat\go_v)p_{n}(v) \notag\\
&\qquad\qquad+\sqrt{n}\sum_{|x(v)|> K\sqrt{n}}(e^{\gb_{n}\tgo_v}-1)p_{n}(v) + \sqrt{n}\gb_n \sum_{|x(v)|\le K\sqrt{n}}(\hat\go_v-\E[\hat\go_v] ) p_{n}(v) 
\end{align}
As  $n$ tends to infinity and then $K\nearrow\infty$ the second term converges to zero, {thanks to the term inside the parenthesis}. Similarly, so does the fourth term thanks to Lemma \ref{Kinfty}.
On the other hand, similarly to previous cases, the sum of the first and third term converge, for $n\to\infty$ and $K\nearrow\infty$, to 
\begin{align*}
&2\int_{\{|w|>1\}\times(0,1)\times\bR }  (e^{\gb w}-1)\gr(t,x) \cP(\dd w \dd t \dd x ) + 2\int_{\{|w|<1\}\times(0,1)\times\bR }  (e^{\gb w}-1-\gb w)\gr(t,x) \cP(\dd w \dd t \dd x )\\
&=2\int_{\bR\times(0,1)\times\bR }  (e^{\gb w}-1-\gb w)\gr(t,x) \cP(\dd w \dd t \dd x ) +2\int_{\{|w|>1\}\times(0,1)\times\bR }  \gb w\gr(t,x) \cP(\dd w \dd t \dd x ).
\end{align*}
These integrals are again well defined thanks to Lemma \ref{PoissonInt}.
Finally, the last term in \eqref{eq:Case3} converges, via the Poisson $L^2(\cP)$ isometry, to
\begin{align*}
2\int_{\{|w|<1\}\times(0,1)\times\bR } w\rho(t,x) \big( \cP(\dd w \dd t \dd x )- \eta(\dd w \dd t \dd x )\big),
\end{align*}
in the limit $n\to\infty$ followed by $K\nearrow\infty$.
This concludes the proof.
\end{proof}
It remains to check:
\begin{lem}\label{Kinfty}
For every $\eps>0$, the following estimate holds true
\begin{align*}
\lim_{K\to\infty}\sup_{n}\,\bP\Big(\sqrt{n}\sum_{|x(v)|> K\sqrt{n}}(e^{\gb_{n}\tgo_v}-1)p_{n}(v) >\eps \Big)= 0
\end{align*}
\end{lem}
\begin{proof}
Denote 
$$\vB^K_j:=[0,n]\times (\,(j-1)K\sqrt{n},\, jK\sqrt{n}\,], \qquad j=1,...,\lceil K^{-1}\sqrt{n} \,\rceil.
$$
We then have, for all large enough $n$, the estimate
\begin{align*}
\bP\Big( & \sqrt{n}\sum_{|x(v)| > K\sqrt{n}}  (e^{\gb_{n}\tgo_v}-1)p_{n}(v) >\eps \Big)\\
&\le \bP\Big(\sqrt{n}\sum_{|x(v)|> K\sqrt{n}}(e^{\gb_{n}\tgo_v}-1) \,\ind_{\go_v\le \gb_n^{-1}} \, p_{n}(v) >\eps/2 \Big)\\ 
&\qquad+ \sum_{j=2}^{\lceil K^{-1}\sqrt{n} \rceil} \bP\Big(\sqrt{n}\sum_{v\in \vB^K_{j}\setminus \vB^K_{j-1}}(e^{\gb_{n}\tgo_v}-1) \, \ind_{\go_v > \gb_n^{-1}} \, p_{n}(v) >\eps 2^{-j-2} \Big).
\end{align*}
The first term can be made arbitrarily small, uniformly in $n$, for $K$ large, since we can use Lemma \ref{lem:comp} to obtain the $L^1(\bP)$ estimate 
\begin{align*}
\sqrt{n} \,\,\,\bE \sum_{|x(v)|> K\sqrt{n}}(e^{\gb_{n}\tgo_v}-1) \,\ind_{\go_v\le \gb_n^{-1}} \, p_{n}(v) &\le
\text{const.}\,\, \bar{F}(\gb_n^{-1}) \sum_{|x(v)|> K\sqrt{n}} \sqrt{n} p_{n}(v)\\
&\le \text{const.}\, n^{-3/2}  \sum_{|x(v)|> K\sqrt{n}} \sqrt{n} p_{n}(v)\\
&\le \text{const.}\, \int_0^1\int_{|x|>K} \rho(t,x) \dd x\dd t.
\end{align*}
Let us now estimate the second term, which
we decompose according to whether there is a large number of sites $ v\in \vB^K_{j}\setminus \vB^K_{j-1}$, such that $\go_v$ exceeds the value $\gb_n^{-1}$ or not. In this way, we arrive at the bound
\begin{align*}
 &\sum_{j=2}^{\lceil K^{-1}\sqrt{n} \rceil} \bP\Big(  \sum_{v\in \vB^K_{j}\setminus \vB^K_{j-1}} \ind_{\go_v>\gb_n^{-1}} \ge K^2j^2 \Big)\\
& + \sum_{j=2}^{\lceil K^{-1}\sqrt{n} \rceil} \bP\Big(\sqrt{n}\sum_{v\in \vB^K_{j}\setminus \vB^K_{j-1}}(e^{\gb_{n}\tgo_v}-1)\, \ind_{\go_v>\gb_n^{-1}}\,p_{n}(v) >\eps 2^{-j-2} \,;\,  \sum_{v\in \vB^K_{j}\setminus \vB^K_{j-1}} \ind_{\go_v>\gb_n^{-1}} < K^2j^2 \Big)\\
&\le K^{-2}\,\bar{F}(\gb_n^{-1}) \sum_{j\ge2} |\vB^K_{j}\setminus \vB^K_{j-1} | \,j^{-2} 
+ \sum_{j\ge2} \bP\Big(K^2j^2\,\sqrt{n} \max_{v\in \vB^K_{j}\setminus \vB^K_{j-1}} \big(\,e^{\gb_n\go_v} p_n(v)\,\big)   > \eps 2^{-j-2}\, \Big)\\
&\le \text{const.}\,K^{-1} \sum_{j\ge2} \,j^{-2} + \sum_{j\ge2} \bP\Big(K^2j^2\, \exp\big(  \gb_n\max_{v\in \vB^K_{j}\setminus \vB^K_{j-1}} \go_v - \text{const.}(K j)^2 \big)\,\big)   > \eps 2^{-j-2}\, \Big),
\end{align*}
where we used the fact that $\bar{F}(\gb_n^{-1}) |\vB^K_{j}\setminus \vB^K_{j-1} |=K\bar{F}(\gb_n^{-1}) n^{3/2} =O(1)$, by the choice of $\gb_n$ and that $\sqrt{n}\max_{v\in \vB^K_{j}\setminus \vB^K_{j-1}} p_n(v)$ converges to the maximum of the heat kernel inside the box $(0,1)\times (K(j-1),Kj]$. A manipulation of the terms inside the last probability, choosing $K$ large enough, leads further to the estimate
 \begin{align*}
 &\text{const.}\,K^{-1} \sum_{j\ge2} \,j^{-2} + \sum_{j\ge2} \bP\Big(   \gb_n\max_{v\in \vB^K_{j}\setminus \vB^K_{j-1}} \go_v \ge \text{const.} (K j)^2 \Big)\\
 &\le \text{const.}\,K^{-1} \sum_{j\ge2} \,j^{-2} + \sum_{j\ge2} \,|\vB^K_{j}\setminus \vB^K_{j-1}\,| \,\bar{F}( \gb_n^{-1} (Kj)^2)\\
 &\le \text{const.}\,K^{-1} \sum_{j\ge2} \,j^{-2} + \text{const.}\,K^{1-2\ga}\,\sum_{j\ge2} j^{-2\ga}.
 \end{align*}
 Notice the crucial fact that, since $\ga>1/2$, the last sum is finite and therefore we can let $K$ be arbitrarily large to show that it is negligible and finally obtain the desired estimate. 
\end{proof}

\appendix
\section{Auxiliary Estimates}\label{sec:aux}

\begin{proof}[Proof of Lemma~\ref{lem:comp}]
{\bf I.} Assuming (i), we can easily establish (ii). Indeed, we have
\begin{align*}
\log \E[e^{\gb_n \go\ind_{\go\le k_n}}]
&=\log \Big(1+ \E[e^{\gb_n \go\ind_{\go\le k_n}}-1]\Big)\\
&=\log \Big(1+\E\big[(e^{\gb_n \go\ind_{\go\le k_n}}-1)\ind_{\go<0}\big]+  \E\big[(e^{\gb_n \go\ind_{\go\le k_n}}-1)\ind_{\go\ge0}\big]\Big)\\
&= \log\Big(1+\E[e^{\gb_n \go_-}-1]+  \E[e^{\gb_n \go_+\ind_{\go\le k_n}}-1]\Big)\\
&=\log\Big(\E[e^{\gb_n \go_-}]+  \E[e^{\gb_n \go_+\ind_{\go\le k_n}}-1]\Big),
\end{align*}
and the result will follow by estimate (i), which we will now establish.
Define the function 
\begin{align*}
	\varphi_p(x):=e^x-\sum_{i=0}^p x^i/i!, \qquad x\in\dR. 
\end{align*}
For any $k>0$ we have
\begin{align*}
	\Big|\E[e^{\gb_n \go_+\ind_{\go\le k}}]-\sum_{i=0}^p \frac{\gb_n^i}{i!} \E[\go_+^i]\Big|
	  & =\abs{\E[\varphi_p(\gb_n\go_+)\ind_{\go_+\le k}]          
	-\sum_{i=1}^{p}\frac{\gb_n^i}{i!} \E[\go_+^i\ind_{\go_+>k}]}\\
	  & \le \E[\abs{\varphi_p(\gb_n\go_+)}\ind_{\go_+\le k}] 
	+\sum_{i=1}^{p}\frac{\gb_n^i k^{i-\theta}}{i!} \E[\go_+^\theta\ind_{\go_+>k}].
\end{align*}
The assumption $\bE[\go_+^\theta]<\infty$ implies that $\E[\go_+^\theta\ind_{\go>k}]\to0$ as $k\to\infty$. 
Thus, 
\begin{align*}
	\sum_{i=1}^{p}\frac{\gb_n^i k^{i-\theta}}{i!} \E[\go_+^\theta\ind_{\go_+>k}]                                 
	&= o\Big(\max\{ (\gb_nk)^{1-\theta}, (\gb_nk)^{p-\theta}\}\cdot \gb_n^\theta\Big),\qquad \text{ as } k\to\infty\\
	&=o(1) \gb_n^\theta,
\end{align*}
since $\gb_n k$ is assumed to stay bounded away form zero and $\theta\in(1,\ga)$.
We now choose arbitrary constant $a>0$, such that $a\gb_n^{-1}<k$. It follows that
\begin{align*}
	\gb_n^{-\theta}\E[\abs{\varphi_p(\gb_n\go_+)}\ind_{\go_+\le a\gb_n^{-1}}]                            
	=\E[\big|(\gb_n\go_+)^{-\theta}\varphi_p(\gb_n\go_+)\big|\ind_{\go_+\le a\gb_n^{-1}}\,\go_+^\theta]=o(1) 
\end{align*}
as $\gb_{n}\to0$, thanks to dominated convergence and the facts that $\theta\in[p,p+1),\lim_{x\to0}x^{-\theta}\varphi_{p}(x)=0$ and $|x^{-\theta}\varphi_p(x)|\le \varphi_p(1)$, for $0\le x\le 1$.  We, now, need to upper bound 
\begin{align*}
	0\le \E[\varphi_p(\gb_n\go_+)\cdot\ind_{a\gb_n^{-1}<\go\le k}]
	  & = \int_{a\gb_n^{-1}}^{k}\gb_n \varphi_{p-1}(\gb_nx)\bar{F}(x)\,\dd x -(\varphi_p(\gb_nk)\bar{F}(k)-\varphi_p(a)\bar{F}(a\gb_n^{-1})) \\
	  & \le \gb_n\int_{a\gb_n^{-1}}^k e^{\gb_nx} \bar{F}(x)\,\dd x +\varphi_p(a)\bar{F}(a\gb_n^{-1})\\
	  & \le \gb_n e^{\gb_n k} \int_{a\gb_n^{-1}}^k  x^{-\ga}L(x)\,\dd x +\varphi_p(a)\bar{F}(a\gb_n^{-1})\\     
	  &=    \gb_n^\ga e^{\gb_n k} \int_{a}^{k\gb_n}  x^{-\ga}L(x\gb_n^{-1})\,\dd x +\varphi_p(a)\bar{F}(a\gb_n^{-1})\\    
	  & \le \text{const.}\, \gb_n^\ga L(\gb_n^{-1}) e^{\gb_n k} \int_{a}^{k\gb_n}  \frac{\dd x}{x^{\ga-\gd}} +\varphi_p(a)\bar{F}(a\gb_n^{-1}),
\end{align*}
where the last inequality is valid for arbitrarily chosen $\gd>0$, such that $\ga-\gd>1$ thanks to Karamata's representation, which provides the estimate $L(x\gb_n^{-1})/L(\gb_n^{-1})<\text{const.} x^\delta$, for all
large $n$ and $x>a$.
 Since $\theta<\ga$, we obtain the estimate of order $o(\gb_n^\theta)$, whenever $k\gb_n$ has the asymptotic behavior prescribed by the assumptions,
 i.e. $e^{\gb_nk_n}\gb_n^{\ga-\theta}\to 0$.
 \vskip 4mm
 {\bf II.} We will only prove the second inequality, since it is more detailed, while the proof of the first is the same.
 We need to distinguish cases:
 \vskip 2mm
 {\bf Case A. ($\ga\in(1/2,1)$)} In this case we have
 \begin{align*}
 \E[e^{\gb_n \go\ind_{|\go|\le k_n}}]-1 &= \E\Big[\big(e^{\gb_n \go}-1\big)\ind_{|\go|\le k_n} \Big]
      =\int_{-k_n}^{k_n} (e^{\gb_n x}-1) \dd F(x)\\
      &= \int_0^{k_n} (e^{\gb_nx}-1) \dd F(x) +  \int_{-k_n}^0 (e^{\gb_nx}-1) \dd F(x) .
 \end{align*}
 Integrating by parts in both integrals, we obtain
 \begin{align*}
 -(e^{\gb_nk_n}-1)\bar{F}(k_n) + \gb_n\int_0^{k_n} e^{\gb_nx}  \bar{F}(x) \dd x  -(e^{-\gb_nk_n}-1)F(-k_n) - \gb_n \int_{-k_n}^0 e^{\gb_nx} F(x) \dd x .
 \end{align*}
Using Assumption \eqref{defC}, this is bounded in absolute value by
 \begin{align*}
 (\gb_n k_n)\,e^{\gb_nk_n}\,\bar{F}(k_n) + \gb_n\int_0^{k_n} e^{\gb_nx}  \bar{F}(x) \dd x   + \gb_n \int_{-k_n}^0 e^{\gb_nx} F(x) \dd x .
 \end{align*}
 The first term clearly satisfies the desired bound. Regarding the second term, we have
  \begin{align*}
  \gb_n\int_0^{k_n} e^{\gb_nx}  \bar{F}(x) \dd x &= \gb_n\int_0^{k_n} e^{\gb_nx}  \, x^{-\ga} L(x) \dd x \le  \gb_n e^{\gb_nk_n} \int_0^{k_n}  x^{-\ga} L(x) \dd x\\
  &\le \text{const.} \gb_n \,e^{\gb_nk_n}\, k_n^{1-\ga} L(k_n) =  \text{const.} \,(\gb_nk_n)\, e^{\gb_nk_n} \bar{F}(k_n).
   \end{align*}
   Similarly, using assumption \eqref{defC} we see that the last term is bounded by $ \text{const.} \,(\gb_nk_n)\, \bar{F}(k_n)$ and this completes the estimate, in this case.
\vskip 2mm
   
 {\bf Case B. ($\ga\in(1,2)$)}  The computation is similar. However, in this case we need to perform a centering
 and for this we use the assumption $\bE[\go]=0$ {(in the second line below)}
 \begin{align*}
 \E[e^{\gb_n \go\ind_{|\go|\le k_n}}]-1 &= \E\Big[\big(e^{\gb_n \go}-1\big)\ind_{|\go|\le k_n} \Big]
      =\int_{-k_n}^{k_n} (e^{\gb_n x}-1) \dd F(x)\\
      &= \int_{-k_n}^{k_n}(e^{\gb_n x}-1-\gb_n x )\dd F(x) -\gb_n\int_{|x|>k_n} x \dd F(x)\\
      &= \int_{0}^{k_n}(e^{\gb_n x}-1-\gb_n x )\dd F(x) +  \int_{-k_n}^0(e^{\gb_n x}-1-\gb_n x )\dd F(x) -\gb_n\int_{|x|>k_n} x \dd F(x).
 \end{align*}
 We will only show how to estimate the integrals over the positive real axis, as the ones over the negative is similar thanks to assumption \eqref{defC}.
 Integration by parts gives
 \[
 \gb_n\int_{k_n}^\infty x \dd F(x)=\gb_nk_n\bar{F}(k_n)+\gb_n\int_{k_n}^\infty x^{-\ga}L(x) \,\dd x\le \text{const.} \,(\gb_nk_n)\, \bar{F}(k_n).
 \]
 Similarly, we have
  \begin{align}\label{eq:CaseB}
  \int_{0}^{k_n}(e^{\gb_n x}-1-\gb_n x )\dd F(x)&=
  -(e^{\gb_n k_n}-1-\gb_n k_n)\bar{F}(k_n)+ \gb_n\int_{0}^{k_n}(e^{\gb_n x}-1 ) \bar{F}(x)\dd x \notag\\
  &\leq (\gb_nk_n)\,e^{\gb_n k_n} \bar{F}(k_n)+ \gb_n\int_{0}^{k_n}(e^{\gb_n x}-1 ) \bar{F}(x)\dd x.
 \end{align}
{Choose, now, $\gd=\min\big\{\gb_nk_n/2, 1\big\}$  and estimate the second term above by}
  \begin{align*}
  \gb_n\int_{0}^{k_n}(e^{\gb_n x}-1 ) \bar{F}(x)\dd x &= \gb_n\int_{0}^{k_n}(e^{\gb_n x}-1 ) \,x^{-\ga} L(x)\dd x \\
  &=  \gb_n\int_{0}^{\gd\gb_n^{-1}}(e^{\gb_n x}-1 ) \,x^{-\ga} L(x)\dd x
   + \gb_n\int_{\gd\gb_n^{-1}}^{k_n}(e^{\gb_n x}-1 ) \,x^{-\ga} L(x)\dd x \\
   &\le \text{const.}\,\gb_n^2\int_{0}^{\gd\gb_n^{-1}} \,x^{1-\ga} L(x)\dd x
   + \gb_n(e^{\gb_n k_n}-1 ) \int_{\gd\gb_n^{-1}}^{k_n} \,x^{-\ga} L(x)\dd x \\
   &\le \text{const.}\Big(\,\gb_n^2 (\gd\gb_n^{-1})^{2-\ga} L(\gd\gb_n^{-1})
   + \gb_n(e^{\gb_n k_n}-1 ) (\gd\gb_n^{-1})^{1-\ga}L(\gd\gb_n^{-1})\,\Big)
   \end{align*}
When $\gb_n k_n$ stays bounded away from zero, then $\gd$ is chosen to be equal to $1$ and this bound reads as $\text{const.} e^{\gb_n k_n} \bar{F}(\gb_n^{-1})$ and 
when $\gb_n k_n$ tends to zero then $\gd=\gb_nk_n/2$ and the last bound writes as $O((\gb_nk_n)^{2-\ga}) \,e^{\gb_n k_n}\,\bar{F}(k_n^{-1})$. Inserting this into \eqref{eq:CaseB} we obtain the desired estimate, since $\alpha<2$.
\end{proof}   
%
\begin{proof}[Proof of Lemma \ref{PoissonInt}]
We only need to check that the random variables $\cW_{\gb}^{(\ga)}$ are well-defined, \ie\ achieve a.s.~finite values. Once this is established, the fact that $\cW_{0}^{(\ga)}$
 has stable distribution with the prescribed characteristic function is standard. 
  To this end, we define the sets 
\begin{align*}
\sA_{1} :=\{(w,x,t): |w| \ge 1+x^2/4\gb t\},
\qquad &\qquad \sA_{2} :=\{(w,x,t): 1\le |w| < 1+x^2/4\gb t\}\\
\text{and } \,\,\,\sA_{3} &:=\{(w,x,t): 0<|w| < 1\}.
\end{align*}
It is easy to check that 
\begin{align*}
\eta(\sA_{1})=\int_{\sA_1} 2\ga w^{-1-\ga}\dd w \dd t \dd x 
&=2 \int_{-\infty}^\infty\!\int_{0}^{1} (1+x^{2}/4\gb t)^{-\ga}\dd x \dd t\\
& = 8 \int_0^{\infty}\!\!\!\int_0^1 \sqrt{t}(1+x^2/\gb)^{-\ga}\dd x \dd t <\infty
\end{align*}
for $\ga>1/2$. This means that a.s. there will only be finite number of points of $\cP$ that take values in $\sA_1$ and therefore 
\[
\int_{\sA_1} e^{\gb w} \rho(t,x)\,\cP(\dd w \dd t \dd x)  <\infty \quad \text{and} \quad \int_{\sA_1} |w|\, \rho(t,x)\,\cP(\dd w \dd t \dd x)  <\infty  \qquad a.s.
\]
Similarly, for every $\ga>1/2$, we have
\[
\int_{\sA_2} e^{\gb w} \rho(t,x)\,\cP(\dd w \dd t \dd x)  <\infty \quad \text{and} \quad \int_{\sA_2} |w|\, \rho(t,x)\,\cP(\dd w \dd t \dd x)  <\infty  \qquad a.s.
\]
since $e^{\gb w} \rho(t,x) \le e^\gb (2\pi t)^{-1/2} $ and $w>1$ on $\sA_2$.
Finally, checking the finiteness of the integrals over $\sA_3$, we have:
first, using the Poisson-$L^2$ isometry we have
 \begin{align*}
 \E^{\cP}\left( \int_{\sA_3} w \rho(t,x) (\cP-\eta)(\dd w \dd t \dd x)  \right)^2 &=  \int_{\sA_3} w^2 \rho(t,x)^2 \eta(\dd w \dd t \dd x)\\
 &= \frac{\ga}{2}\int_{\sA_3} w^2 \rho(t,x)^2 \frac{(\ind_{w>0}+c_-\ind_{w<0})\dd w}{|w|^{1+\ga}} \dd x \dd t <\infty,
 \end{align*}
 for $\ga\in[1,2)$, while for the same values of $\ga\in[1,2)$ we have
 \begin{align*}
\E^{\cP} 
\int_{ \sA_3} |e^{\gb w}-1-\gb w|\rho(t,x)\cP(\dd w \dd t \dd x)  &=
\int_{\sA_3} |e^{\gb w}-1-\gb w|\rho(t,x) \eta(\dd w \dd t \dd x)\\
&\le \text{const.} \int_{\sA_3} w^2 \rho(t,x) \,\frac{(\ind_{w>0}+c_-\ind_{w<0})\dd w}{|w|^{1+\ga}} \dd x \dd t <\infty.
\end{align*}
Finally,
 \begin{align*}
 \E^{\cP}\int_{\sA_3} |w| \rho(t,x) \,\cP(\dd w\dd x \dd t) &= \int_{\sA_3} |w| \rho(t,x)\, \eta(\dd w\dd x \dd t) \\
 &= \frac{\ga}{2}\int_{\sA_3} |w| \rho(t,x) \, \frac{(\ind_{w>0}+c_-\ind_{w<0})\dd w}{|w|^{1+\ga}} \dd x \dd t<\infty,
 \end{align*}
 for $\ga<1$.
The claim now follows by putting the above estimates together.

To evaluate the integrals explicitly, let us define
\[
\psi_{\ga}(y):= 
\begin{cases}
 \int_{\sS} (e^{iy w \rho(t,x)}-1-iy w \rho(t,x)) \,\, \eta(\dd w \dd x\dd t)\, & \text{ if } \ga\in (1,2) \\
\int_{\sS\cap \{|w|>1 \} } (e^{iy w \rho(t,x)}-1)  \eta(\dd w \dd x\dd t)                                    & \\
\qquad\qquad+ \int_{\sS \cap \{|w|\le 1\}} (e^{iy w \rho(t,x)}-1-iy w \rho(t,x)) \,\, \eta(\dd w \dd x\dd t)\,  &\text{ if } \ga=1 \\
\int_{\sS} (e^{iy w \rho(t,x)}-1)  \eta(\dd w \dd x\dd t)\, & \text{ if } \ga\in (0,1) ,   
\end{cases}
\]
It is easy to see that $\psi_{\ga}(-y)=\overline{\psi_{\ga}(y)}$. Assuming, w.l.o.g. that $y>0$ and restricting to the case $\ga\in (0,1)$, we have
\begin{align*}
\psi_{\ga}(y) &= y^{\ga}\cdot \int_{\dR} (\ind_{w>0}+c_-\ind_{w<0})(e^{iw}-1) \frac{\ga\,\dd w}{|w|^{1+\ga}}\cdot  \int_0^1\!\!\!\int_{\dR} \frac12\rho(t,x)^{\ga} \dd x\dd t.
\end{align*}
Using the result that $\int_{0}^{\infty} (e^{iw}-1) \frac{\ga\,\dd w}{w^{1+\ga}} = -\Gamma(1-\ga)e^{-i \ga\pi/2}$ for $\ga\in (0,1)$ and $\int_{0}^{1}\!\!\int_{\dR} \frac12\rho(t,x)^{\ga} \dd x\dd t = \frac{(2\pi)^{(1-\ga)/2}}{(3-\ga)\sqrt{\ga}}$ we finally have
\begin{align*}
\psi_{\ga}(y) &= - |y|^{\ga} \cdot \frac{(2\pi)^{(1-\ga)/2}\cos(\ga\pi/2)\Gamma(1-\ga)}{(3-\ga)\sqrt{\ga}} \cdot \bigg((1+c_-) - i\sgn(y)(1-c_{-})\tan\frac{\ga\pi}{2}\bigg).
\end{align*}
One can also explicitly evaluate $\psi_{\ga}(y)$ for $\ga\in[1,2)$ in a similar fashion. 
\end{proof}
	
\section*{Acknowledgements}
We are grateful to Xue-Mei Li and Pierre Le Doussal for very useful discussions that initiated this work. This work was done while the first author was a visiting Assist.~Professor at the Univ.~of Warwick. NZ also thanks Academia Sinica where part of this work was completed.

\bibliographystyle{amsplain}
\bibliography{}

\bigskip
\begin{thebibliography}{AAAAAa}
\bibitem[AKQ10]{AKQ10}
T.~Alberts, K.~Khanin, and J.~Quastel,
The Intermediate Disorder Regime for Directed Polymers in Dimension 1+1.
J., Physical Review Letters, 
{\em J., Phys. Rev. Lett.} 105, 090603. (2010)

 
 \bibitem[AKQ12]{AKQ12}
T.~Alberts, K.~Khanin, and J.~Quastel,
Intermediate disorder regime for $1+1$ dimensional directed polymers.
{\em Ann.\ Probab.} (to appear), arXiv:1202.4398.

\bibitem[ACQ11]{ACQ11}
G.~Amir, I.~Corwin, J.~Quastel, 
Probability distribution of the free energy of the continuum directed random polymer in 1 + 1 dimensions.
{\em  Comm. Pure Appl. Math.}, 64:466Ð537 (2011).

\bibitem[AL11]{AL11}
A.~Auffinger, O.~Louidor,
Directed polymers in random environment with heavy tails.  
{\em Comm. on Pure and Applied Math. } Vol. 64 no. 2, 183-204 (2011) 

\bibitem[BBP07]{BBP07}
G.~Biroli, J.P.~Bouchaud, M.~Potters,
Extreme value problems in Random Matrix Theory and other disordered systems
{\em J. Stat. Mech.} (2007) P07019

\bibitem[BC14]{BC14}
A.~Borodin, I.~Corwin,
Macdonald Processes. 
{\em Prob. Th. Rel. Fields}, 158:225Ð 400 (2014).

\bibitem[BCR13]{BCR13}
A.~Borodin, I.~Corwin, D.~Remenik,
Log-Gamma polymer free energy fluctuations via a Fredholm determinant identity. 
{\em Comm. Math. Phys.}, 324:215Ð232 (2013).

\bibitem[COSZ14]{COSZ14}
 I.~Corwin, N.~O'Connell, T.~Sepp\"al\"ainen, N.~Zygouras,
Tropical combinatorics and Whittaker functions
{\em Duke Math J.}, (2014) Vol. 163, Number 3, 513-563

\bibitem[CSZ15]{CSZ15}
F.~Caravenna, R.~Sun, N.~Zygouras,
Polynomial chaos and scaling limits of disordered systems,
{\em J. Eur. Math. Soc.}, to appear

\bibitem[DZ02]{DZ02}
A.~Dembo, O.~Zeitouni,
Large deviations and applications. 
{\em Handbook of stochastic analysis and applications}, Textbooks Monogr., 163, (2002)

\bibitem[FNS77]{FNS77}
D.~Forster, D.R.~Nelson, M.J.~Stephen, 
Large-distance and long-time properties of a randomly stirred fluid,
{\em Phys.~Rev.~A} 16:732-749 (1977)

\bibitem[GLBR14]{GLBR14}
T.~Gueudr\'e, P.~Le Doussal, J-P.~Bouchaud, A.~Rosso
Revisiting Directed Polymers with heavy-tailed disorder
{\em  arXiv:1411.1242}

\bibitem[HM07]{HM07}
B.~Hambly, J.B.~Martin,
Heavy tails in last-passage percolation., 
{\em Prob. Th. Rel. Fields }137,no. 1-2, 227Ð275 (2007)

\bibitem[HH85]{HH85}
D. A. ~Huse, C. L.~Henley, 
Pinning and roughening of domain walls in Ising systems due to random impurities
{\em Phys. Rev. Lett.} 54:2708-2711 (1985).

\bibitem[HHF85]{HHF85}
D.~Fisher, D.A.~Huse, C.L.~Henley, 
{\em Phys. Rev. Lett.} 55:2094 (1985)

\bibitem[FSV14]{FSV14}
G.~Moreno-Flores, T.~Sepp\"al\"ainen, B.~Valk\'o,
Fluctuation exponents for directed polymers in the intermediate disorder regime
 arXiv:1312.0519

\bibitem[K02]{K02}
O.~Kallenberg,
Foundations of modern probability.
{\em Springer-Verlag, New York}, 2002

\bibitem[M91]{M91}
C.~Mueller,
On the support of solutions to the heat equation with noise. 
{\em Stochastics Stochastics Rep.}, 37(4):225Ð245, 1991

\bibitem[LLH83]{LLR83}
M.R.~Leadbetter, G.~Lindgren, H.~ RootzŽn,
 Extremes and related properties of random sequences and processes. 
 {\em Springer Series in Statistics}. Springer-Verlag, (1983)
 
 \bibitem[O12]{O12}
 N.~O'Connell
 Directed polymers and the quantum Toda lattice. 
{\em Ann. Probab.} 40 (2012) 437-458. pdf
 
 \bibitem[OSZ14]{OSZ14}
 N.~O'Connell, T.~Sepp\"al\"ainen, N.~Zygouras,
 Geometric RSK correspondence, Whittaker functions and symmetrized random polymers
 {\em Invent. Math.  } Vol. 197, Issue 2, 361-416 (2014)
 
 \bibitem[SS10]{SS10}
 T.~Sasamoto, H.~Spohn.
 One-dimensional KPZ equation: an exact solution and its universality. 
 {\em Phys. Rev. Lett.,} 104:23 (2010)
 
 \bibitem[S12]{S12}
 T.~Sepp\"al\"ainen,
 Scaling for a one-dimensional directed polymer with boundary conditions
 {\em Annals of Probability} (2012), Vol. 40, No. 1, 19-73
 \end{thebibliography}

\end{document}